\documentclass[11pt]{article}

\usepackage{graphicx}
\usepackage{amssymb}
\usepackage{array}
\usepackage{amsthm}
\usepackage{mathtools}
\usepackage{tgtermes}
\usepackage{enumitem}
\usepackage{mathrsfs}
\usepackage{stmaryrd}
\usepackage[all]{xy}
\usepackage{verbatim} 
\usepackage[hidelinks]{hyperref}
\usepackage{fullpage}  
\usepackage{hyperref}
\usepackage{setspace}
\usepackage{tikz-cd}
\usepackage{amsmath}
\usetikzlibrary{positioning}
\usepackage{cleveref}
\usepackage{makecell}

\newcommand{\RR}{\mathbb{R}}
\newcommand{\ZZ}{\mathbb{Z}}
\newcommand{\PP}{\mathbb{P}}

\newcommand{\CC}{\mathbb{C}}
\newcommand{\HH}{\mathbb{H}}
\newcommand{\QQ}{\mathbb{Q}}

\newcommand{\EE}{\mathbb{E}}

\newcommand{\FF}{\mathbb{F}}

\newcommand{\DD}{\mathbb{D}}
\newcommand{\C}{\mathbb{C}}
\newcommand{\Z}{\mathbb{Z}}
\newcommand{\CP}{\mathbb{CP}}
\newcommand{\R}{\mathbb{R}}
\newcommand{\Q}{\mathbb{Q}}

\DeclareMathOperator{\Mod}{Mod}

\DeclareMathOperator{\Aut}{Aut}
\DeclareMathOperator{\GL}{GL}
\DeclareMathOperator{\PGL}{PGL}
\DeclareMathOperator{\Tr}{Tr}
\DeclareMathOperator{\OO}{O}

\DeclareMathOperator{\Stab}{Stab}

\DeclareMathOperator{\Fix}{Fix}

\DeclareMathOperator{\Bl}{Bl}

\DeclareMathOperator{\MOD}{Mod}

\DeclareMathOperator{\Ric}{Ric}
\DeclareMathOperator{\Homeo}{Homeo}
\DeclareMathOperator{\Diff}{Diff}
\DeclareMathOperator{\Isom}{Isom}

\DeclareMathOperator{\RRef}{Ref}

\DeclareMathOperator{\Id}{Id}

\newtheorem{theorem}{Theorem}[section]
\newtheorem{proposition}[theorem]{Proposition}
\newtheorem{lemma}[theorem]{Lemma}
\newtheorem{corollary}[theorem]{Corollary}

\theoremstyle{definition}
\newtheorem{definition}[theorem]{Definition}

\newtheorem{question}[theorem]{Question}

\newtheorem{remark}[theorem]{Remark}

\numberwithin{equation}{subsection}
\title{Cyclic Nielsen realization for del Pezzo surfaces}
\author{Seraphina Eun Bi Lee, Tudur Lewis, and Sidhanth Raman}
\date{}

\begin{document}
\maketitle
\begin{abstract}
    The cyclic Nielsen realization problem for a closed, oriented manifold asks whether any mapping class of finite order can be represented by a homeomorphism of the same order.  In this article, we resolve the smooth, metric, and complex cyclic Nielsen realization problem for certain ``irreducible" mapping classes on the family of smooth 4-manifolds underlying del Pezzo surfaces. Both positive and negative examples of realizability are provided in various settings. Our techniques are varied, synthesizing results from reflection group theory and 4-manifold topology.
\end{abstract}

\setcounter{tocdepth}{1}
\tableofcontents

\section{Introduction}\label{sectionIntro}
For $0 \leq n \leq 8$, consider the closed, oriented, smooth 4-manifold
$$M_n := \CC \PP^2 \# n \overline{\CC\PP^2}$$
and let $\Mod(M_n) := \pi_0(\Homeo^+(M_n))$ denote its \emph{topological} mapping class group. The manifold $M_n$ admits both the structures of a complex surface (including del Pezzo surfaces) and Einstein manifolds (including (conformally) K\"ahler--Einstein metrics). These structures provide a rich source of homeomorphisms and mapping classes of $M_n$ through their automorphisms.

One approach to studying the elements of $\Mod(M_n)$ is through the (topological) \emph{Nielsen realization problem}, which asks whether every finite subgroup $G \leq \Mod(M_n)$ has a lift $\tilde G \leq \Homeo^+(M_n)$ isomorphic to $G$ under the natural quotient map $q: \Homeo^+(M_n) \to \Mod(M_n)$. In the classical case of closed, oriented surfaces $\Sigma_g$, $g \geq 2$, Kerckhoff's solution \cite{kerckhoff} of the Nielsen realization problem shows that finite subgroups of $\Mod(\Sigma_g)$ are exactly those that arise as automorphisms of complex structures (and equivalently, isometries of hyperbolic structures) on $\Sigma_g$. In the case of cyclic subgroups $G \leq \Mod(\Sigma_g)$, the Nielsen realization problem was resolved earlier by Nielsen \cite{nielsen} and Fenchel \cite{fenchel}. 

In analogy with the case of surfaces, we study certain finite cyclic subgroups $\langle f \rangle \cong \Z/m\Z \leq \Mod(M_n)$ and compare which of these subgroups are realizable by finite groups of diffeomorphisms or automorphisms of a complex or metric structure. In other words, we compare the solutions to the following refinements of the Nielsen realization problem arising from each type of geometric structure on $M_n$ for certain elements $f \in \Mod(M_n)$ of order $m < \infty$:
\begin{itemize}
    \item \textit{(Complex Nielsen)} Does there exist a complex structure $J$ on $M_n$ and a complex automorphism $\varphi \in \Aut(M_n, J)$ of order $m$ such that $[\varphi] = f$?
    \item \textit{(Metric Nielsen)} Does there exist an Einstein metric $g$ on $M_n$ and an isometry $\varphi \in \Isom(M_n, g)$ of order $m$ such that $[\varphi] = f$?
    \item \textit{(Smooth Nielsen)} Does there exist a diffeomorphism $\varphi \in \Diff^+(M_n)$ of order $m$ such that $[\varphi] = f$?
\end{itemize}
If yes, we say that such a mapping class $f$ is \emph{realizable} by a complex (resp. metric, smooth) automorphism of $M_n$.

\subsection{Statement of results}

In this paper we study certain finite-order \emph{irreducible} elements $f \in \Mod^+(M_n)$, an index-$2$ subgroup of $\Mod(M_n)$ (see \Cref{sectionTopDelPezzo}). An element $f \in \Mod(M_n)$ is called irreducible if $f$ does not preserve any decomposition $H_2(M, \Z) \oplus H_2(\# k \overline{\CP^2}, \Z) \cong H_2(M_n, \Z)$ where $M$ is a del Pezzo manifold and $k > 0$; see \Cref{irreddef} for a more precise statement. The realizability of irreducible elements can be used to deduce the realizability of some reducible elements (cf. \Cref{primenielsen}).

The failure of complex Nielsen realization for certain finite-order, irreducible $f \in \Mod(M_n)$ was already known by work of Dolgachev--Iskovskikh \cite{dolgachev--iskovskikh}, and in this paper we find many such classes that are not smoothly realizable (see \Cref{ComplexVSSmooth}). However, the following theorem shows that in many cases, the three types of Nielsen realization problems considered above have equivalent solutions.

Below, we say that $(M_n,J)$ is a \emph{del Pezzo surface} if $J$ is a complex structure on $M_n$ that is biholomorphic to a non-singular surface with ample anticanonical bundle. A metric $g$ is an \emph{Einstein} metric if its Ricci curvature tensor is proportional to $g$. See Section \ref{sectionMetComplex} for the precise definition of Einstein metrics of positive, symplectic type.

\begin{theorem}[Equivalence of Nielsen Realizations]\label{mainthm1}
Let $3 \leq n \leq 7$ and let $f\in \MOD^+(M_n)$ be irreducible and of order $m < \infty$. Suppose further that $H_2(M_n, \Z)^{\langle f \rangle} \cong \Z$. The following are equivalent: 
\begin{enumerate}[label=(\alph*)]
    \item (Metric) There exists an Einstein metric $g$ (of positive symplectic type) on $M_n$ and an isometry $\varphi \in \Isom(M_n,g)$ of order $m$ such that $[\varphi] = f$.
    \item (Complex) There exists a complex structure $J$ so that $(M_n, J)$ is a del Pezzo surface and a complex automorphism $\varphi \in \mathrm{Aut}(M_n,J)$ of order $m$ such that $[\varphi] = f$.
    \item (K\"{a}hler--Einstein) There exists an K\"{a}hler--Einstein pair $(M_n,g,J)$ and an automorphism $\varphi \in \Aut(M_n,g,J)$ of order $m$ such that $[\varphi]  =f$.
    \item (Smooth) There exists a diffeomorphism $\varphi \in \Diff^+(M_n)$ of order $m$ such that $[\varphi] = f$.
\end{enumerate}    
\end{theorem}

\begin{remark}[On the condition $H_2(M_n, \Z)^{\langle f \rangle} \cong \Z$]
\Cref{irredclassification} shows that if $f \in \Mod^+(M_n)$ (with $3 \leq n \leq 8$) is irreducible of finite order then $H_2(M_n, \Z)^{\langle f \rangle} \cong \Z$ or $f$ fixes a nontrivial class $v \in H_2(M_n, \Z)$ of self-intersection $0$. The study of classes $f$ satisfying the first condition is closely related to that of automorphisms of del Pezzo surfaces; \Cref{irred-minimal} and \cite[Theorem 3.8]{dolgachev--iskovskikh} or \cite[Proposition 4.1]{blanc} together imply that if $f$ is realizable by a complex automorphism $\varphi \in \Aut(M_n, J)$ then $(M_n, J)$ is in fact a del Pezzo surface. 

Moreover, this condition on $H_2(M_n, \Z)^{\langle f \rangle}$ cannot be removed in general; see \Cref{Jonquiere--not--einstein} for an explicit example of an irreducible class $f \in \Mod^+(M_7)$ with $H_2(M_7, \Z)^{\langle f \rangle} \not\cong\Z$ that satisfies (d) but not (a), (b) or (c).

With the connection to del Pezzo surfaces in mind, we choose to focus on irreducible classes $f$ with $H_2(M_n, \Z)^{\langle f \rangle} \cong \Z$ in \Cref{mainthm1}. 
\end{remark}

The proof of \Cref{mainthm1} is partly synthetic and partly enumerative. A simple synthetic proof of equivalence between metric, complex, and K\"{a}hler--Einstein Nielsen realization for the mapping classes of $M_n$ considered in \Cref{mainthm1} is given in \Cref{MCequiv}. Proving that complex and smooth Nielsen realization are equivalent requires more work. Using the Coxeter theory of a distinguished finite subgroup $W_n \leq \Mod^+(M_n)$ called the \emph{Weyl group} (\Cref{defn:weyl}), we reduce the problem to obstructing the smooth realizability of finitely many conjugacy classes of $W_n$. We then study each such conjugacy class carefully in \Cref{ComplexVSSmooth}. 

While the existence of finite, non-cyclic subgroups of $\Mod(M_n)$ that do not lift to $\Diff^+(M_n)$ was known for $n= 2, 3$ by work of the first-named author \cite[Corollary 1.4, Theorem 1.6]{lee}, \Cref{mainthm1} differs from these previous results as it shows that there exist individual, finite-order elements of $\Mod(M_n)$ that are not smoothly realizable. As such, the proof of the equivalence of (b) and (d) of \Cref{mainthm1} requires more refined obstructive tools, such as an analysis (\Cref{sec:A7}) of the topology of possible finite quotients of $M_n$.

Note that there are other del Pezzo manifolds not covered by \Cref{mainthm1}, such as $M_n$ for $n = 0, 1, 2$ and $\CC\PP^1 \times \CC\PP^1$. Nielsen realization for irreducible mapping classes in these cases is straightforward to verify --- see \Cref{blowup2more}. We also give partial positive Nielsen realization results in the last remaining case of $n=8$, e.g. \Cref{coxetermainthm}. The following question remains open.
\begin{question}\label{ques1}
    Are the metric, complex, K\"ahler--Einstein, and smooth Nielsen realization problems equivalent for finite, cyclic subgroups $G = \langle f \rangle \leq \Mod(M_8)$ generated by irreducible classes $f$?
\end{question}

In addition to Theorem \ref{mainthm1}, we provide explicit examples of realizable and non-realizable irreducible elements of $\MOD(M_n)$ by a more refined study of the Coxeter theory of the Weyl group $W_n$. For example, the Coxeter elements form a distinguished conjugacy class of $W_n$ of order $h_n$, the Coxeter number of $W_n$ (see \Cref{coxeterDef}). In the following theorem, we show that Coxeter elements also form a distinguished conjugacy class of $\Mod^+(M_n)$ among classes of order $h_n$ for all $3 \leq n \leq 8$.
\begin{theorem}[Coxeter Nielsen Realization]\label{coxetermainthm}
    Let $3 \leq n \leq 7$ and let $f \in \mathrm{Mod}^+(M_n)$ be an irreducible mapping class of order $h_n$. Then $f$ is conjugate to a Coxeter element and is realizable as a automorphism of order $h_n$ on a del Pezzo surface $(M_n, J)$. For $n=8$, suppose $f \in \mathrm{Mod}^+(M_8)$ is irreducible of order $h_8$.
\begin{itemize}
    \item If $f \in \Aut(H_2(M_8, \Z), Q_{M_8})$ has trace $0$ then $g$ is conjugate to a Coxeter element, and is realizable by a complex automorphism of order $h_8$ on a del Pezzo surface $(M_8, J)$.
    \item If $f \in \Aut(H_2(M_8, \Z), Q_{M_8})$ has nonzero trace then $g$ is not realizable by any finite-order diffeomorphism of $M_8$. 
\end{itemize}
\end{theorem}

Coxeter elements account for large numbers of realizable elements in $W_n$, as recorded below:
\begin{center}
\begin{tabular}{|c|>{\centering\arraybackslash}p{1.5cm}|>{\centering\arraybackslash}p{1.5cm}|>{\centering\arraybackslash}p{1.5cm}|>{\centering\arraybackslash}p{1.5cm}|>{\centering\arraybackslash}p{1.5cm}|>{\centering\arraybackslash}p{1.5cm}|>{\centering\arraybackslash}p{1.5cm}|}
\hline
$n$ & 3 & 4&5&6&7&8\\ \hline\hline
\makecell{ realizable, irreducible\\ elements of order $h_n$ }  & 2       & 24 & 240                                    & 4320 &161280   &23224320\\ \hline
\end{tabular}
\end{center}

Coxeter elements also characterize irreducible mapping classes of odd, prime order. 
\begin{theorem}\label{thm:primes-are-powers}
Let $3 \leq n \leq 8$ and let $f \in \Mod^+(M_n)$ be an irreducible mapping class of odd, prime order. Then $f$ is conjugate in $\Mod^+(M_n)$ to a power of a Coxeter element of $W_n \leq \Mod^+(M_n)$. 
\end{theorem}
Theorem \ref{thm:primes-are-powers} follows from Theorem \ref{primeIrredCuspidal}, which gives a more refined characterization of irreducible elements of odd, prime order. The coarser fact that any irreducible $g \in \Mod^+(M_n)$ of odd, prime order is realizable by a complex automorphism can also be deduced from the enumerative proof of Corollary \ref{CarterClassify}, which follows from the work of Dolgachev--Iskovskikh \cite{dolgachev--iskovskikh} and Carter \cite{carter}.

Theorem \ref{thm:primes-are-powers} sheds light on the automorphisms realizing the prime order irreducible classes via the birational maps of $\CP^2$ that lift to Coxeter elements. These maps take the form $(x,y) \mapsto (a,b) + (y,y/x)$ in affine coordinates, for certain choices of $(a,b) \in \C^2$, so they can be studied explicitly; see \cite[Section 11]{mcmullen} for more details.

Finally, we consider the complex realizability of classes $f \in \Mod(M_n)$ (both irreducible and reducible) of odd, prime order. Note that Theorems \ref{coxetermainthm} and \ref{thm:primes-are-powers} together handle the realizability of irreducible elements of odd, prime order. In the following corollary, we give an inductive argument to deduce complex realizability of reducible classes using the irreducible case as a base case. 
\begin{corollary}\label{primenielsen}
Let $0 \leq n \leq 8$. If $f \in \MOD(M_n)$ has odd prime order $p$ then $f$ is realizable  by a complex automorphism of $(M_n, J)$ for some complex structure $J$ turning $M_n$ into an $n$-fold blowup of $\CP^2$.
\end{corollary}

Our proof of \Cref{primenielsen} relies on the enumerative result of \Cref{thm:primes-are-powers}, and it would be interesting to find moduli-theoretic (or other non-enumerative) proofs of \Cref{mainthm1} or \Cref{primenielsen} in full generality. This seems possible in some special cases. For example, McMullen's work \cite{mcmullen} on dynamics of automorphisms of $M_n$ (which mainly focuses on the dynamics of $M_{n}$ for $n \geq 10$) can be applied to show that irreducible elements of odd, prime order are realizable by complex automorphisms. See \Cref{rmk:mcmullen} for more details.

\subsection{Related work}

The classical Nielsen realization problem for surfaces $\Sigma_g$ was solved by Kerckhoff \cite{kerckhoff} for all finite subgroups $G \leq \Mod(\Sigma_g)$; the special case of finite cyclic subgroups $G$ was solved earlier by Nielsen \cite{nielsen} and Fenchel \cite{fenchel}. In higher dimensions, Raymond--Scott \cite{raymond--scott} showed that the Nielsen realization problem fails in every dimension $d \geq 3$ in certain nilmanifolds. In dimension $4$, recent work of Farb--Looijenga \cite{farb--looijenga} and the first-named author \cite{lee, lee-involutions} address variants of the Nielsen realization problem for K3 surfaces and del Pezzo surfaces respectively. Baraglia--Konno \cite{baraglia2023note} and Konno \cite{konno2022dehn} have used essentially orthogonal techniques, e.g. Seiberg--Witten invariants, to prove non-realization results, the latter of which extends some of the results in \cite{farb--looijenga} to spin $4$-manifolds with nonzero signature. Arabadji--Baykur \cite{arabadji--baykur} and Konno--Miyazawa--Taniguchi \cite{konno--miyazawa--taniguchi} have since also obtained non-realization results for non-spin $4$-manifolds.

The Nielsen realization problem for del Pezzo surfaces is intricately linked to the study of finite subgroups of the plane Cremona group. The general case is studied completely in work of Dolgachev--Iskovskikh \cite{dolgachev--iskovskikh}; also see work of Blanc \cite{blanc} on the classification of finite-order conjugacy classes of the Cremona group. Our current work relies crucially on the classification of minimal rational $G$-surfaces for finite cyclic groups $G$ in \cite{dolgachev--iskovskikh}. Many special cases were studied previously as well; for example, see Bayle--Beauville \cite{bayle--beauville} and Beauville--Blanc \cite{beauville--blanc} for the case of $G = \Z/p\Z$ with $p$ prime.

Birational automorphisms of $\CP^2$ have also been used to generate interesting examples of diffeomorphisms of rational surfaces outside the finite-group setting. For example, McMullen \cite{mcmullen} used the action of $W_n$ on the moduli space of marked blowups of $\CP^2$ without nodal roots to realize Coxeter elements of $W_n \leq \Mod(M_n)$ by dynamically-interesting complex automorphisms. Many other people have also studied the dynamics of automorphisms coming from the Cremona group, such as Bedford--Kim \cite{bedford--kim2006,bedford--kim2009} who also found rational surface automorphisms of positive entropy, and Cantat--Lamy \cite{cantat--lamy} who constructed normal subgroups of the Cremona group via its action on the Picard--Manin space. 

Finally, we point out that some tools (\Cref{irredclassification}) developed in this paper can be used to simplify some proofs of \cite{lee-involutions} by the first-named author, by eliminating much of the casework by arguments in the style of Bayle--Beauville \cite{bayle--beauville}. See \Cref{rmk:avoiding-mori-thy} and \Cref{geiserbertini} for more details.

\subsection{Outline}

In \Cref{sectionTopMapping}, we give some background on the topological mapping class group of del Pezzo manifolds and explicate its relationship to Coxeter theory. In \Cref{sectionRefBlowup}, we provide connections of our problem to algebraic geometry via minimal $G$-surfaces. Most notably, we give various homological critera that irreducible mapping classes must satisfy, and prove a reduction result regarding the possible conjugacy classes of complex \textit{non-realizable} irreducible mapping classes (\Cref{CarterClassify}). In \Cref{sectionMetComplex}, we resolve the first three equivalences of \Cref{mainthm1}. In \Cref{ComplexVSSmooth}, we leverage a variety of techniques (branched covers of 4-manifolds, $G$-signature theorem, etc.) to prove smooth non-realizability in all possible cases outlined in \Cref{CarterClassify}. This completes the proof of \Cref{mainthm1} which is given in \Cref{sec:mainthmproof}. In \Cref{sectionCoxeter}, we analyze explicit examples of (non-)realizability of distinguished finite-order classes from the view of Coxeter theory, and as a consequence prove Theorems \ref{coxetermainthm} and \ref{thm:primes-are-powers}. Using these two theorems, we deduce a proof of Corollary \ref{primenielsen} as well.

\subsection{Acknowledgements}

We thank our advisors, Benson Farb, Tara Brendle and Vaibhav Gadre, and Jesse Wolfson respectively, for their support, interest, and feedback on this work. We thank Farb for organizing a workshop on 4-manifold mapping class groups, where we all first met and would subsequently begin working on this project together. SL thanks Cindy Tan for many helpful discussions about branched coverings of $4$-manifolds and Carlos A. Serv\'an for insightful conversations. SL also thanks Igor Dolgachev and Allan Edmonds for helpful email correspondences about their respective previous work. TL thanks Sam Lewis and Riccardo Giannini for insightful conversations. SR thanks Josh Jordan and Jeff Streets for helpful discussions.

\section{Coxeter theory of $\Mod(M_n)$}\label{sectionTopMapping}

This section recalls known facts about the topological mapping class group, $\Mod(M_n)$, of the del Pezzo manifold $M_n$. We introduce the key property of \emph{irreducibility} for certain mapping classes studied in this paper, and give it a Coxeter theoretic interpretation that is crucial in the enumeration proofs of Section \ref{sectionRefBlowup}.

\subsection{Topological mapping class groups of del Pezzo manifolds}\label{sectionTopDelPezzo}

Del Pezzo surfaces are smooth complex algebraic surfaces with ample anti-canonical bundles. The diffeomorphism type of every del Pezzo surface is either $\CC\PP^2$, $\CC\PP^1 \times \CC\PP^1$, or $M_n = \CC\PP^2 \# n \overline{ \CC\PP^2}$ for $1 \leq n \leq 8$. We refer to these smooth manifolds as \emph{del Pezzo manifolds}. By the Mayer--Vietoris sequence,
$$H_k(M_n,\ZZ) = \begin{cases}
\ZZ & \text{ if }k=0,4,\\
0 & \text{ if }k=1,3,\\
\ZZ^{n+1} &\text{ if } k=2.
\end{cases}$$
The intersection form $Q_{M_n}$ on $H_2(M_n,\ZZ)$ takes the form $\langle 1\rangle \oplus n\langle -1 \rangle$ with respect to the natural $\Z$-basis $\{H, E_1, \dots, E_n\}$ of $H_2(M_n; \Z)$, where $H$ denotes the hyperplane class and $E_1, \dots, E_n$ denote the $n$ exceptional divisors. Let $\ZZ^{1,n}$ denote the lattice $\ZZ^{n+1}$ equipped with the signature $(1,n)$-form $\langle 1\rangle \oplus n\langle -1 \rangle$.

By theorems of Freedman \cite{freedman}, Perron \cite{perron1986pseudo}, Quinn \cite{quinn}, Cochran--Habegger \cite{cochran--habegger}, and Gabai--Gay--Hartman--Krushkal--Powell \cite{gabai2023pseudo}, there is an isomorphism
\begin{equation}\label{mod-isomorphism}
    \MOD(M_n) := \pi_0(\Homeo^+(M_n)) \to \Aut(H_2(M_n,\ZZ), Q_{M_n}) \cong \OO(1,n)(\ZZ).
\end{equation}

\begin{remark}
Surjectivity of (\ref{mod-isomorphism}) is due to Freedman \cite{freedman}. Injectivity follows from work of Perron \cite{perron1986pseudo}, Quinn \cite{quinn}, Cochran--Habegger \cite{cochran--habegger}, and Gabai--Gay--Hartman--Krushkal--Powell \cite{gabai2023pseudo}. See \cite[Section 1.3]{gabai2023pseudo} for a more detailed history of this theorem.
\end{remark}

We identify $\MOD(M_n)$ with $\OO(1,n)(\ZZ)$ using the isomorphism (\ref{mod-isomorphism}). Let $\MOD^+(M_n)$ denote the subgroup
\[
    \OO^+(1,n)(\ZZ) \leq \OO(1,n)(\ZZ),
\]
that preserves each of the two sheets of the hyperboloid $\{v \in H_2(M_n,\RR): Q_{M_n}(v,v) = 1\}$. Each connected component (equipped with the form $-Q_{M_n}$) of the hyperboloid is isometric to hyperbolic space $\HH^n$. Recall the notion of irreducibility for a mapping class.

\begin{definition}[{\cite[Definition 1.1]{lee-involutions}}]\label{irreddef}
A mapping class $g \in \MOD(M_n)$ is called \textit{reducible} if there exists a del Pezzo manifold $M$ and some $k > 0$ such that there is an isometry
$$\iota : (H_2(M,\ZZ) \oplus H_2(\#k\overline{\CC\PP^2} , \ZZ),Q_M \oplus Q_{\#k\overline{\CC\PP^2}}) \to (H_2(M_n,\ZZ), Q_{M_n})$$
with $g$ is contained in the image of the induced inclusion $$\iota_* : \MOD(M)\times \MOD(\# k \overline{\CC\PP^2}) \hookrightarrow \MOD(M_n).$$ Otherwise, $g$ is called \textit{irreducible}.
\end{definition}

There are a few reasons to restrict attention to irreducible mapping classes. Finite order reducible mapping classes that are Nielsen realizable are often given by a $G$-equivariant connect sum of irreducible mapping classes or non-minimal rational $G$-surfaces (see \cite[Corollary 1.5]{lee-involutions} and \Cref{primenielsen}). Another more algebro-geometric reason is the connection between irreducibility and the classical notion of minimal rational $G$-surfaces (see \Cref{irred-minimal}).

\subsection{$\OO^+(1,n)(\Z)$ is a Coxeter group for $2 \leq n \leq 9$.}\label{sec:coxeter}
For each vector $v \in \ZZ^{1,n}$ with $Q_{M_n}(v,v) = \pm 1, \pm 2$, the \emph{reflection} $\RRef_v \in \OO(1,n)(\Z)$ \emph{about $v$} is the map
\[
    \RRef_v: x \mapsto x - \frac{2Q_{M_n}(x,v)}{Q_{M_n}(v,v)} v.
\]
Wall \cite{wall-indefinite-orthogonal} determined explicit generators for $\OO^+(1,n)(\Z)$ for $2 \leq n \leq 9$ in terms of reflections.
\begin{theorem}[{Wall 1964, \cite[Theorems 1.5, 1.6]{wall-indefinite-orthogonal}}]\label{thm:wall-generators}
The groups $\OO^+(1,n)(\Z)$ have generators:
\begin{align*}
    \OO^+(1,2)(\Z) &= \langle \RRef_{H-E_1-E_2}, \, \RRef_{E_1-E_2}, \, \RRef_{E_2} \rangle, \\
     \OO^+(1,n)(\Z) &= \langle \RRef_{H-E_1-E_2-E_3}, \, \RRef_{E_1-E_2}, \, \RRef_{E_2-E_3}, \, \dots, \, \RRef_{E_{n-1}-E_n}, \, \RRef_{E_n} \rangle \text{ if }3 \leq n \leq 9.
\end{align*}
\end{theorem}
For an analysis of (a finite-index subgroup of) $\OO^+(1,n)(\Z)$ as a hyperbolic reflection group, see Vinberg \cite{vinberg}. Consider the Coxeter system $(W(n), S(n))$ specified in \Cref{fig:o1n-coxeter}. There is a surjective homomorphism $\rho: W(n) \to \OO^+(1,n)(\Z)$ given on the generators by
\[
    \rho(s_0) = \begin{cases}
    \RRef_{H - E_1 - E_2} & \text{ if }n= 2, \\
    \RRef_{H - E_1 - E_2 - E_3} & \text{ if }3 \leq n \leq 9,
    \end{cases} \qquad \rho(s_k) = \begin{cases}
    \RRef_{E_k - E_{k+1}} & \text{ if }1 \leq k \leq n-1, \\
    \RRef_{E_n} & \text{ if } k = n.
    \end{cases}
\]
Let $\Phi_n: \OO^+(1,n)(\Z) \to \OO(\Z^{1,n} \otimes \R)$ denote the natural action of $\OO^+(1,n)(\Z)$ on $\Z^{1,n} \otimes \R$ and let $\Psi_n: W(n) \to \GL(V_n)$ be the geometric representation of $W(n)$ (cf. \cite[Section 5.3]{humphreys}). There is an isometry $(V_n, B_n) \to \Z^{1,n} \otimes \R$ (where $B_n$ is the bilinear form defined in \cite[Section 5.3]{humphreys}) such that under this identification, the geometric representation $\Psi_n$ coincides with the composition $\Phi_n \circ \rho$ (\cite[Section 2.4]{lee}). Because $\Psi_n$ is injective (\cite[Corollary 5.4]{humphreys}), the homomorphism $\rho$ is in fact an isomorphism. Hence $\OO^+(1,n)(\Z)$ is a Coxeter group.

\begin{figure}
    \centering
    \includegraphics{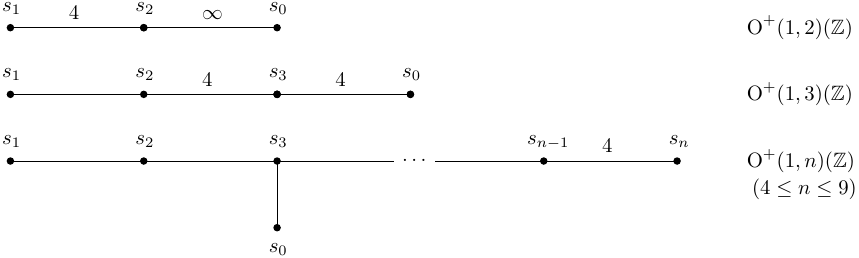}
    \caption{Coxeter diagrams for $(W(n), S(n))$. The subgraph spanned by the vertices $s_0, \dots, s_{n-1}$ is the Coxeter diagram for the Weyl group $W_n.$ }
    \label{fig:o1n-coxeter}
\end{figure}

\subsection{Weyl group}

The \emph{Weyl group} $W_n \leq \Mod(M_n)$ is a distinguished Coxeter subgroup of $\Mod^+(M_n)$.
\begin{definition}\label{defn:weyl}
    The \textit{Weyl group} $W_n \leq \MOD(M_n)$ is the stabilizer of the (Poincar\'e dual of the) canonical class
$$K_{M_n} := -3H + \sum_{k=1}^n E_k \in H_2(M_n, \ZZ).$$ 
\end{definition}
For $3\leq n \leq 8$, the group $W_n$ is a finite group and is generated by the \textit{simple reflections} (\cite[Corollary 8.2.15]{dolgachev})
\begin{equation}\label{eqn:weyl-generators}
W_n = \langle \RRef_{H - E_1 - E_2 - E_3} \text{, }\RRef_{E_1 - E_2}\text{, }\RRef_{E_2 - E_3} \text{, }\dots\text{, }\RRef_{E_{n-1} - E_{n}}\rangle
\end{equation}
and $W_n$ naturally acts on the lattice
\[
	\EE_n := \{v \in \Z^{1,n} : Q_{M_n}(v, K_{M_n}) = 0\} = \Z\{H - E_1 - E_2 - E_3, \, E_1 - E_2, \, \dots, \, E_{n-1}-E_n\}.
\]

There are isomorphisms
\begin{align*}
    W_1 & \cong 1, &W_2 & \cong \Z/2\Z, &W_3 & \cong W(A_2)\times W(A_1), \\
    W_4 & \cong W(A_4), &W_5 & \cong W(D_5), &W_i & \cong W(E_i), \; 6 \leq i \leq 8,
\end{align*}
where $W(\Gamma)$ denotes the Coxeter group with Coxeter diagram $\Gamma$. This can be seen by drawing the Coxeter diagrams for the simple reflections as in \Cref{fig:o1n-coxeter}, and using the classification of finite reflection groups.

\subsection{Irreducibility, parabolic subgroups, and cuspidal conjugacy classes} \label{sectionCuspidal}

We now introduce some tools from Coxeter theory to characterize the irreducible elements of $\Mod^+(M_n)$. Let $W := W(S)$ be a Coxeter group with generating set $S$. For any $I \subset S$, consider the \emph{parabolic subgroup} 
$$W_I := \langle s : s \in I \rangle \subset W(S).$$
The group $W_I$ is again a Coxeter group with generating set $I$ \cite[Theorem 5.5]{humphreys}.

Coxeter groups admit special conjugacy classes called \emph{cuspidal classes} that are useful in the classification of irreducible mapping classes. 
\begin{definition}
A conjugacy class $C$ of a Coxeter group $W$ is called a \textit{cuspidal class} if $C \cap W_I = \emptyset$ for all proper subsets $I \subset S$. For a finite Coxeter group equipped with the geometric representation, the stabilizer of a nonzero vector is a conjugate of a proper parabolic subgroup \cite[Theorem 5.13, Exercise 5.13]{humphreys}, hence an element generates a cuspidal conjugacy class if and only if it does not fix a nonzero vector; cf. \cite[Lemma 3.1.10]{geck2000characters}. 
\end{definition}

From now on, we specialize to the Coxeter theory of $W := \Mod^+(M_n)$ for $3 \leq n \leq 8$ and its parabolic subgroups. The group $W$ has as simple roots
\[
    \Delta_n := \{H-E_1-E_2-E_3, \, E_1-E_2, \, E_2 - E_3, \, \dots \, E_{n-1} - E_n, \, E_n\}.
\]
By (\ref{eqn:weyl-generators}), the Weyl group $W_n$ is a (finite) parabolic subgroup of $W$. The Weyl group will be one of two important parabolic subgroups of $W = \Mod^+(M_n)$ appearing in this paper:
\[
    W_n = W_{\Delta_n - \{E_n\}}, \qquad P_n := W_{\Delta_n - \{E_1-E_2, E_n\}}.
\]
The Coxeter diagram of $P_n$ coincides with that of $W(D_{n-1})$, and so $P_n \cong W(D_{n-1})$. As shown below, many parabolic subgroups $W_I$ of $\Mod^+(M_n)$ are finite. 

The following lemma provides a necessary condition for irreducibility.
\begin{lemma}[{Proof of \cite[Lemma 2.6(2)]{lee-involutions}}]\label{lem:irred-criterion-from-involutions}
Let $w \in W$. Let $I \subseteq \Delta_n - \{\alpha\}$ where $\alpha \neq E_1 - E_2, \, E_n$. If $w$ is conjugate in $W$ to an element of the parabolic subgroup $W_{I}$ then $w$ is reducible.
\end{lemma}
\begin{proof}
If $\alpha \neq E_1 - E_2, \, E_n$ then $\alpha \in \Delta_n$ is one of:
\[
    H- E_1-E_2-E_3 \, \text{ or }\, E_\ell - E_{\ell+1} \text{ for }2 \leq \ell \leq n-1. 
\]
For each choice of $\alpha$ above, the proof of \cite[Lemma 2.6(2)]{lee-involutions} explicitly finds a del Pezzo surface $M$, an integer $k > 0$, and an isometric decomposition
\[
    (H_2(M_n; \Z),\, Q_{M_n}) \cong (H_2(M; \Z) \oplus H_2(\# k \overline{\CP^2}; \Z),\, Q_M \oplus Q_{\# k \overline{\CP^2}})
\]
which is preserved by $W_{\Delta_n - \{\alpha\}}$. Hence every element of $W_I$ is reducible.
\end{proof}

Parabolic subgroups arise naturally as the stabilizers of points under the action of Coxeter groups on their Tits cones. Let $R_n := - Q_{M_n}$ denote the signature $(n,1)$ form on $H_2(M_n, \R)$. Vinberg's algorithm applied to $W$ \cite[Proposition 4, Table 4]{vinberg}, gives a (closed) fundamental domain of the action of $W$ on $\HH^n \subseteq H_2(M_n, \R)$ as the polyhedron
\[
	D := \bigcap_{\alpha \in \Delta_n} \{v \in \HH^n : R_n(v, \alpha) \leq 0\}
\]
where $\Delta_n$ denotes the set of simple roots of $W$. (To be precise, the fundamental domain $D$ above is $g(P)$, where $g = \RRef_{E_1} \circ \dots \circ \RRef_{E_n} \in W$ and $P$ is the fundamental domain found by Vinberg.) The representation $W \to \OO(H_2(M_n; \R), R_{n})$ is isomorphic to the geometric representation of $W$ (cf. \Cref{sec:coxeter}). By unimodularity of $R_n$, the pairing $R_n$ induces an isomorphism $H_2(M_n, \R) \to H_2(M_n, \R)^*$ of representations of $W$ where $H_2(M_n, \R)^*$ denotes the contragredient representation of the geometric representation. Let $U \subseteq H_2(M_n, \R)$ be the Tits cone of $W$, where $H_2(M_n, \R)$ is identified with $H_2(M_n, \R)^*$ via $R_n$. Then $D$ is contained in $-U$. Moreover, $\mathbb H^n \subseteq -U$ as well because because the $W$-translates of $D \subseteq \mathbb H^n$ cover $\mathbb H^n$. The $W$-stabilizer of any point $p \in D$ is equal to the $W$-stabilizer of the corresponding point $-p \in U$. 

Define the subset $D_I \subseteq D$ for any $I \subseteq \Delta_n$
\[
	D_I := \left(\bigcap_{\alpha \in I} \{v \in \HH^n: R_n(v, \alpha) = 0\}\right) \cap \left(\bigcap_{\alpha \notin I} \{v \in \HH^n: R_n(v, \alpha) < 0\}\right)
\]
so that the subsets $D_I$ exhaust $D$. By \cite[Theorem 5.13]{humphreys}, the stabilizer of any point in $D_I$ in $W$ is the parabolic subgroup $W_I$. Note that for any pair of subsets $I \subseteq J \subseteq \Delta_n$, there is an inclusion of closed faces $\overline {D_J} \subseteq \overline{D_I}$ and an inclusion of stabilizers $W_I \subseteq W_J$. Moreover, $\overline{D_{I \cup J}} = \overline{D_I} \cap \overline{D_J}$, and hence for any $p \in D$, there exists a largest subset $I \subseteq \Delta_n$ such that $p$ is contained in the face $\overline{D_I}$. 

The following lemma is the key technical tool used to characterize irreducible elements of $\Mod^+(M_n)$ in terms of conjugacy classes of $W_n$ and $P_n$.
\begin{lemma}\label{irredclassification}
Let $3 \leq n \leq 8$. Let $w \in \MOD^+(M_n)$ have finite order and consider its fixed subspace $H_2(M_n, \ZZ)^{\langle w \rangle}$. If $w$ is irreducible then up to conjugacy in $\MOD^+(M_n)$, either 
\begin{enumerate}[label=(\alph*)]
    \item $H_2(M_n, \ZZ)^{\langle w \rangle} =\ZZ\{K_{M_n}\}$ (or equivalently, $w$ generates a cuspidal conjugacy class of $W_n$), or \label{DelPezzoCase}
    \item $H_2(M_n, \ZZ)^{\langle w \rangle} = \ZZ\{K_{M_n}, H-E_1\}$ (or equivalently, $w$ generates a cuspidal conjugacy class of $P_n$).\label{ConicCase}
\end{enumerate}
\end{lemma}
\begin{proof}
Let $w \in \Mod^+(M_n)$ be irreducible and have finite order. Because $w$ has finite order, $w$ fixes at least one point $\mathbb H^n$. After possibly conjugating $w$ by an element of $W$, we may assume that $w$ fixes at least one point $p \in D \subseteq \mathbb H^n$. 

Let $I \subseteq \Delta_n$ be any subset so that $w$ fixes some point in $D_I$. Because all points of $D_I$ have equal stabilizer $W_I$, equivalently we may take $I \subseteq \Delta_n$ so that $w$ fixes all points of $D_I$ and hence $\overline{D_I}$.

Suppose that $\Delta_n - \{E_1 - E_2\} \subseteq I$ so that $D_I \subseteq \overline{D_{\Delta_n - \{E_1 - E_2\}}}$. Compute that the space
\[
    \left(\bigcap_{\alpha \neq E_1 - E_2} \{v \in H_2(M_n; \R) : R_n(v, \alpha) = 0\}\right) \cap \{v \in H_2(M_n; \R) : R_n(v, E_1 - E_2) \leq 0\}
\]
consists of nonnegative scalar multiples of $H - E_1$ so that
\[
    \overline{D_{\Delta_n - \{E_1 - E_2\}}} = \R_{\geq 0}\{H - E_1\} \cap \mathbb H^n = \emptyset.
\]
Therefore $D_I = \emptyset$, and so we may assume that $\Delta_n - \{E_1 - E_2\}$ is not contained in $I$. In particular, $I \neq \Delta_n$ and there exists some $\alpha \in \Delta_n$ so that $I$ is contained in $\Delta_n - \{\alpha\}$. If $\alpha \neq E_n$ and $\alpha \neq E_1 - E_2$ then $w$ is reducible by Lemma \ref{lem:irred-criterion-from-involutions}. All together, this implies that
\[
    \{H - E_1 - E_2 - E_3, \, E_2 - E_3, \, \dots, \,  E_{n-1}-E_n\} \subseteq I.
\]
In other words, if $w$ is contained in $W_I$ then 
\[
    I = \Delta_n - \{E_n\} \quad \text{ or }\quad \Delta_n - \{E_n, \, E_1-E_2\}.
\]

In either case, $w$ is contained in $W_n$, fixes 
\[
    D_{\Delta_n - \{E_n\}}= \{p = -(9-n)^{-\frac 12}K_{M_n}\},
\]
and the action of $w$ on $H_2(M_n; \R)$ preserves the decomposition
\[
    H_2(M_n, \R) \cong \R\{K_{M_n}\} \oplus \R\{K_{M_n}\}^\perp.
\]
The vector space $\R\{K_{M_n}\}^\perp$ has a basis
\[
    \{H - E_1 - E_2 - E_3, \, E_1 - E_2, \, \dots, \, E_{n-1}-E_n\},
\]
so the action of $W_n$ on $\R\{K_{M_n}\}^\perp$ is isomorphic to the geometric representation of $W_n$. 

If $w$ does not fix any nonzero elements of $\R\{K_{M_n}\}^\perp$ then \ref{DelPezzoCase} holds, as
\[
    H_2(M_n, \Z)^{\langle w \rangle} = \Z\{K_{M_n}\}
\]
and $w$ represents a cuspidal conjugacy class of $W_n$ by \cite[Lemma 3.1.10]{geck2000characters}.

Suppose now that $w$ fixes a nonzero $y \in \R\{K_{M_n}\}^\perp$. Because $W_n$ is finite, the Tits cone of $W_n$ is the whole space $\R\{K_{M_n}\}^\perp$ (where the geometric representation $\R\{K_{M_n}\}^\perp$ is identified with the contragredient representation $(\R\{K_{M_n}\})^*$ via $R_n$, whose restriction to $\R\{K_{M_n}\}^\perp$ is positive-definite) \cite[Exercise 5.13]{humphreys}. The stabilizer of any nonzero point in the Tits cone of $W_n$ is conjugate to a proper parabolic subgroup of $W_n$ \cite[Theorem 5.13]{humphreys}, and so $w$ is contained in a proper parabolic subgroup $(W_n)_J$ of $W_n$, up to conjugacy in $W_n$. Because $(W_n)_J = W_{J}$ is a proper parabolic subgroup of $W$ containing $w$, we see that $w$ is contained in $P_n$ up to conjugacy in $W_n$.

Because each root $\beta \in \Delta_n - \{E_n, E_1 - E_2\}$ satisfies $R_n(H - E_1, \beta) = 0$,
\[
    \R\{K_{M_n}, H - E_1\} \subseteq H_2(M_n, \R)^{\langle w \rangle}
\]
and the action of $w$ on $H_2(M_n, \R)$ preserves the decomposition
\[
    H_2(M_n, \R) \cong \R\{K_{M_n},\, H-E_1\} \oplus \R\{K_{M_n},\, H-E_1\}^\perp.
\]
The vector space $\R\{K, H-E_1 \}^{\perp}$ has a basis 
\[
    \{E_2-E_3, \, \dots ,\, E_{n-1}-E_n,\, H-E_1-E_2-E_3\},
\]
so the action of $P_n$ on $\R\{K, H-E_1 \}^{\perp}$ coincides with the geometric representation of the Coxeter group $P_n$. Because $P_n$ is finite, the Tits cone of $P_n$ is the whole space $\R\{K, H-E_1 \}^{\perp}$ (again, where the geometric representation $\R\{K_{M_n}, H-E_1 \}^\perp$ is identified with the contragredient representation via $R_n$, whose restriction to $\R\{K_{M_n}, H-E_1 \}^\perp$ is positive-definite) \cite[Exercise 5.13]{humphreys}. The stabilizer of a nonzero point of $\R\{K, H-E_1 \}^{\perp}$ is conjugate to a proper parabolic subgroup of $P_n$ \cite[Theorem 5.13]{humphreys}. Any proper parabolic subgroup of $P_n$ is also a proper parabolic subgroup of $W$, and so $w$ does not fix any nonzero point of $\R\{K, H-E_1 \}^{\perp}$.

Therefore \ref{ConicCase} holds, as 
\[
    H_2(M_n, \Z)^{\langle w \rangle} = \Z\{K_{M_n}, H - E_1\}
\]
and $w$ represents a cuspidal conjugacy class of $P_n$ by \cite[Lemma 3.1.10]{geck2000characters}.
\end{proof}

\begin{remark}\label{rmk:avoiding-mori-thy}
\Cref{irredclassification} should be seen as an analog of \cite[Theorem 3.8]{dolgachev--iskovskikh} (cf. Section \ref{sec:irred-min}). We point out that its proof only uses the hyperbolic reflection group structure of $\MOD^+(M_n)$ for small $n$ and avoids the Mori theory used in the algebro-geometric arguments of \cite[Theorem 3.8]{dolgachev--iskovskikh}, \cite[Lemma 1.2]{bayle--beauville}, or \cite[Proposition 4.1]{blanc}. We use \Cref{irredclassification} to simplify the enumeration of irreducible mapping classes similarly as how Mori theory is used in classification of minimal rational $G$-surfaces. In terms of the action of $w$ on $H_2(M_n, \ZZ)$, the two possibilities in \Cref{irredclassification} correspond to the del Pezzo surface and conic bundle cases of \cite[Theorem 3.8]{dolgachev--iskovskikh}, respectively.
\end{remark}

The following lemma is a partial converse to \Cref{irredclassification}.
\begin{lemma}
    \label{posdefcriterionirred}
    Let $3 \leq n \leq 8$ and let $w \in \MOD^+(M_n)$. If $H_2(M_n,\ZZ)^{\langle w \rangle} = \ZZ\{K_{M_n}\}$ then $w$ is irreducible.
\end{lemma}
\begin{proof}
    Suppose $w$ is reducible and that $w(K_{M_n}) = K_{M_n}$. Then there are elements $v_1,..,v_r \in H_2(M_n,\ZZ)$ with $Q_{M_n}(v_i , v_j )= -\delta_{ij}$ for all $1 \leq i,j \leq r$, such that $w (\Z \{v_1,..,v_r\}) = \Z \{v_1,..,v_r\}$. For each $i \in \Z$, there exists $1 \leq j \leq r$ so that $w^i(v_1) = \pm v_j$ because $w$ restricts to an isometry of $\Z\{v_1,..,v_r\}$.

    Let $k$ be the smallest positive integer for which $w^k(v_1) = v_1$ and let $v = \sum_{j=0}^{k-1} w^j(v_1)$ so that $w(v) = v$. If $v=0$ then $w^j(v_1)=-v_1$ for some $0 \leq j \leq k-1$ by linear independence of the set $\{v_i\}$, so $v_1$ is contained in $\mathbb E_n = \Z\{K_{M_n}\}^\perp$ because
    \[
        Q_{M_n}(v_1 , K_{M_n}) = Q_{M_n}(w^j(v_1) , w^j(K_{M_n}) )= -Q_{M_n}(v_1 , K_{M_n}) = 0.
    \]
    However, $Q_{M_n}(v_1, v_1) = -1$ and $\EE_n$ is an even lattice by \cite[Lemma 8.2.6]{dolgachev}. Therefore $v \neq 0$. 
    
    Moreover, $v$ is not contained in $\Z\{K_{M_n}\}$ because $Q_{M_n}(v, v) < 0$ but $Q_{M_n}(K_{M_n}, K_{M_n}) = 9-n > 0$. Therefore, $H_2(M_n, \Z)^{\langle w \rangle} \neq \Z\{K_{M_n}\}$.
    \end{proof}
    
On the other hand, the converse to \Cref{irredclassification} does not always hold.
\begin{lemma}\label{4and6iff}
    If $n=  4,6$, an element $w \in \MOD^+(M_n)$ is irreducible if and only if up to conjugacy in $\MOD^+(M_n)$, the fixed subspace satisfies $H_2(M_n,\ZZ)^{\langle w \rangle} = \ZZ\{K_{M_n}\}$ and $w$ generates a cuspidal conjugacy class of $W_n$.
\end{lemma}
\begin{proof}
    By Lemmas \ref{irredclassification} and \ref {posdefcriterionirred}, it suffices to show that if $n = 4, 6$ and if $H_2(M_n, \Z)^{\langle w \rangle} = \Z\{K_{M_n}, H - E_1\}$ then $w$ is reducible. To see this, it suffices to find some $v \in H_2(M_n,\ZZ)^{\langle w \rangle}$ with $Q_{M_n}(v,v) = \pm 1$. For $n= 6$, let $v = K_{M_6} + (H- E_1)$. For $n = 4$, let $v = K_{M_4} + (H- E_1)$. 
\end{proof}

\subsection{Cuspidal conjugacy classes of $P_n$}

In this section we examine the cuspidal conjugacy classes of $P_n \cong W(D_{n-1})$ following the notation of Carter \cite[Section 7]{carter}. The group $W(D_{n-1})$ acts on a set $\{e_1, -e_1, \dots, e_{n-1}, -e_{n-1}\}$, where $\{e_1, \dots, e_{n-1}\}$ is an orthonormal basis of the geometric representation of $W(D_{n-1})$. Denote by $\psi: W(D_{n-1}) \to S_{n-1}$ the natural quotient given by recording the permutation on the orthonormal basis and ignoring negations. For an element $g \in W(D_{n-1})$, let $(k_1 ,\dots , k_r)$ be a cycle of $\psi(g) \in S_{n-1}$. Then $g\in W(D_{n-1})$ acts on $e_{k_1}, \dots, e_{k_r}$ by
\[
    e_{k_1} \mapsto \pm e_{k_2} \mapsto \cdots \mapsto \pm e_{k_r} \mapsto \pm e_{k_1}.
\]
\begin{definition}\label{defn:cycle-type}
    The cycle $(k_1 ,\dots ,k_r)$ is called \textit{positive} if $g^r(e_{k_1}) = e_{k_1}$ and \textit{negative} if $g^r(e_{k_1}) = -e_{k_1}$. The lengths of the cycles of $g$ and their signs define the \textit{signed cycle-type} of $g$. A positive $k$-cycle is denoted $[k]$ and a negative $k$-cycle is denoted $[\Bar{k}]$. A product of disjoint signed cycles is denoted $[k_1k_2\dots k_m \Bar{\ell}_1\Bar{\ell}_2\dots \Bar{\ell}_{m'}]$.
\end{definition} 

The following lemma parametrizes cuspidal conjugacy classes of $W(D_{n-1})$ by signed cycle types.
\begin{lemma}[{cf. \cite[Proposition 3.4.11]{geck2000characters}}]\label{negativeEven}
There is a bijection between the set of unordered, even partitions $\alpha = (\alpha_1, \dots, \alpha_r)$ of $n-1$ (i.e. partitions with $r$ even) and the set of cuspidal conjugacy classes of $W(D_{n-1})$ given by 
\[
    (\alpha_1, \dots, \alpha_r) \mapsto \text{conjugacy class of } w_\alpha^-,
\]
where $w_\alpha^-\in W(D_{n-1})$ is some element with signed cycle-type $[\Bar{\alpha}_1 \dots \Bar{\alpha}_r]$. The order of $w_\alpha^-$ is $\mathrm{lcm}_{1 \leq j \leq r}(2\alpha_j)$. 
\end{lemma}
\begin{proof}
The bijection is established in \cite[Lemma 3.4.10, and Proof of Theorem 3.2.7 for $D_n$]{geck2000characters}. It remains to compute the order of $w_{\alpha}^-$. According to \cite[Section 3.4.3]{geck2000characters}, the characteristic polynomial of $w_{\alpha}^-$ with respect to the geometric representation of $W(D_{n-1})$ is given by
\[
\mathrm{det}(t I_{\RR^{n-1}}-w_{\alpha}^-) = \prod_{j=1}^r(t^{\alpha_j}+1).
\]
Any eigenvalue of $w_{\alpha}^-$ in the geometric representation of $W(D_{n-1})$ is a $(2\alpha_j)$th root of unity for some $1 \leq j\leq r$, and any primitive $(2\alpha_j)$th root of unity is an eigenvalue of $w_{\alpha}^-$. Therefore, $w_{\alpha}^-$ has order $\textrm{lcm}_{1 \leq j \leq r}\{2\alpha_j\}$.
\end{proof}

\section{Complex Nielsen realization via enumeration}\label{sectionRefBlowup}

In this section, we record known facts about complex automorphisms of rational surfaces. The goal is to deduce Corollary \ref{CarterClassify}, which enumerates the irreducible mapping classes of $M_n$ that are not realizable by complex automorphisms of a del Pezzo surface. 

\subsection{Irreducibility and minimality}\label{sec:irred-min}

A result of Friedman--Qin \cite[Corollary 0.2]{friedman--qin} implies that any complex structure on the smooth manifold $M_n$ is that of a rational surface, so the study of complex Nielsen realization problem for $M_n$ reduces to that of complex automorphisms of rational surfaces. 

Consider the following classical notion from the study of birational automorphisms of $\CC\PP^2$ of finite order. For reference, see \cite[Section 3]{dolgachev--iskovskikh}.

\begin{definition}
Let $G$ be a finite group. A \textit{rational $G$-surface} is a pair $(S,\rho)$ where $S$ is a rational surface and $\rho : G \to \Aut(S)$ is an injective homomorphism.

A \textit{minimal} rational $G$-surface is a rational $G$-surface $(S,\rho)$ such that any birational morphism of $G$-surfaces $(S,\rho) \to (S',\rho')$ is an isomorphism.
\end{definition}

\begin{lemma}\label{irred-minimal}
Let $n \leq 8$ and let $w \in W_n \leq \MOD^+(M_n)$ be irreducible and let $G = \langle w \rangle$. Suppose $w$ is realizable by a complex automorphism $\varphi$ of some complex structure $(M_n, J)$. Then $(M_n, J, \langle\varphi\rangle)$ is a minimal rational $G$-surface.
\end{lemma}
\begin{proof}
Let $\varphi \in \Aut(M_n, J)$ such that $[\varphi] = w$ for some complex structure $J$ on $M_n$ turning $(M_n,J)$ into a rational $G$-surface (not necessarily an iterated blowup of $\CC\PP^2$).

In order to prove the claim, we need to show that if there exists a birational morphism $b : M_n \to S'$ such that $b \circ \varphi = \psi \circ b$ for some $\psi \in \Aut(S')$, then $b$ is an isomorphism. If $b$ is not an isomorphism then $b$ is an iterated blowup by \cite[Theorem II.11]{beauville}. There is an identification
\begin{equation}\label{eq241}
    H_2(M_n,\ZZ) \cong H_2(S',\ZZ) \oplus \ZZ\{e_1,\dots,e_m\}
\end{equation}
where $e_1,\dots,e_m$ are the exceptional divisors of the iterated blowup $b : M_n \to S'$. Even though the divisors $E_k$ are not necessarily pairwise disjoint, their $\ZZ$-span is isometric to 
$$(\ZZ\{e_1,\dots,e_m\} , Q_{M_n}|_{\ZZ\{e_1,\dots,e_m\}}) \cong (\ZZ^m, m\langle -1 \rangle).$$
Now note that $\varphi$ preserves the decomposition (\ref{eq241}) and the lattice $(H_2(S', \Z), Q_{S'})$ is isometric to $\Z^{1,n-m}$ or to $(H_2(\CP^1 \times \CP^1, \Z), Q_{\CP^1 \times \CP^1})$ by \cite[Theorem 1.2.21]{gompf--stipsicz} because the rank of $H_2(S', \Z)$ is $1 + n - m \leq 8$ and $(H_2(S', \Z), Q_{S'})$ has signature $1-(n-m)$. In other words, $w$ is reducible if $m > 0$. Hence if $w$ is irreducible then $b$ is an isomorphism, and $(M_n, J,\langle w \rangle)$ is a minimal $G$-surface.
\end{proof}

\begin{remark}
The minimal rational $G$-surfaces (for $G$ cyclic) are classified by Dolgachev--Iskovskikh \cite{dolgachev--iskovskikh} and Blanc \cite{blanc}. However, it is not clear how to extract a classification of all \textit{irreducible} mapping classes in $\MOD(M_n)$ from the classification without enumerating all conjugacy classes of $W_n$ using Carter graphs (see \Cref{CarterSection}) or via computer software.
\end{remark}

According to \cite[Theorem 3.8]{dolgachev--iskovskikh} (cf. \cite[Proposition 4.1]{blanc}), Mori theory shows that if $(S, \rho)$ is a minimal rational $G$-surface then one of two possibilities occur: 
\begin{enumerate}[label=(\alph*)]
\item The surface $S$ is a del Pezzo surface and $H_2(S,\ZZ)^G \cong \ZZ$.\label{item:delpezzo}
\item The surface $S$ admits the structure of a conic bundle and $H_2(S,\ZZ)^G \cong \ZZ^2$. Here, a \textit{conic bundle structure} on a rational $G$-surface $(S, \rho)$ is a $G$-equivariant morphism $\varphi: S \to \CC\PP^1$ such that the fibers are isomorphic to a reduced conic in $\CC\PP^2$.\label{item:conic}
\end{enumerate}
Therefore the study of diffeomorphisms of $M_n$ arising from complex automorphisms reduces to studying the complex automorphisms of del Pezzo surfaces and conic bundles. 

\subsection{Carter graphs and complex non-realizability}\label{CarterSection}

The goal of this section is to enumerate the Carter graphs associated to irreducible mapping classes that are not realizable by complex automorphisms of del Pezzo surfaces. To do so, we compare the list of automorphisms of del Pezzo surfaces to the list of \emph{Carter graphs} \cite{carter} parameterizing the conjugacy classes of $W_n$. We recall the construction of these graphs below.

By \cite[p. 5, p. 45]{carter}, every element $w \in W_n$ can be written as a product of two involutions $w = w_1w_2$ so that the $(-1)$-eigenspaces $V_{-1}(w_i)$ of $w_i$, $i =1, 2$, acting on $\mathbb E_n \otimes \R$ have trivial intersection $V_{-1}(w_1) \cap V_{-1}(w_2) = 0$. Moreover, each such involution can be written as a product of reflections 
$$w_1 = \RRef_{\alpha_1} \cdots \RRef_{\alpha_k} \text{ and } w_2 = \RRef_{\alpha_{k+1}} \cdots \RRef_{\alpha_{k+m}}$$
where $\{\alpha_1,\dots,\alpha_k\}$ and $\{\alpha_{k+1},\dots,\alpha_{k+m}\}$ are two sets of mutually orthogonal roots with respect to the bilinear form $Q_{M_n}$ restricted to $\EE_n$ by \cite[Lemma 5]{carter}.

\begin{definition}
    Define a graph $\Gamma = (V,E)$ (with respect to the above factorization of $w$) where the set of vertices $V$ correspond to the set of roots $\{\alpha_1,\dots,\alpha_{k+m}\}$, and two distinct vertices $\alpha_i$ and $\alpha_j$ are joined by $k$-many edges $e \in E$, where 
    $$k = \frac{2Q_{M_n}(\alpha_i, \alpha_j)}{Q_{M_n}(\alpha_i,\alpha_i)} \cdot \frac{2Q_{M_n}(\alpha_j,\alpha_i)}{Q_{M_n}(\alpha_j,\alpha_j)}.$$
    Such a graph $\Gamma$ is called a \emph{Carter graph}. We follow the notation convention of \cite[Sections 4 and 5]{carter} for labelling Carter graphs and the convention of \cite[Sections 7 and 8]{carter} assigning to each conjugacy class a single associated Carter graph associated to some factorization of a representative element.
\end{definition}

Dolgachev--Iskovskikh \cite[Table 9]{dolgachev--iskovskikh} analyze the conjugacy classes of $W_n$ generated by elements $w \in W_n$ that are contained in the image of the natural map $\Aut(M_n, J) \to \Mod(M_n)$ for some del Pezzo surface $(M_n, J)$. Conversely, the image of $\Aut(M_n, J) \to \Mod(M_n)$ is contained in $W_n$ (up to conjugacy in $\Mod(M_n)$ for any del Pezzo surface $(M_n, J)$. To see this, note that any del Pezzo surface $(M_n, J)$ admits a birational morphism $\pi: (M_n, J) \to \CP^2$ which is an $n$-fold iterated blowup and specifies $n$-many exceptional divisors (cf. \cite[p. 355]{dolgachev} or \cite[Section 6.1]{dolgachev--iskovskikh}). Up to the action of $\Mod(M_n)$, the homology classes of the exceptional divisors are given by $E_1,\dots,E_n \in H_2(M_n,\Z)$, and the class representing a line in $\CP^2$ disjoint from each $E_i$ is given by $H \in H_2(M_n,\Z)$. Let $c_1(M_n,J) \in H^2(M_n,\Z)$ denote the first Chern class of $(M_n,J)$. Its Poincar\'e dual is sent to $K_{M_n}$ by some element of $\Mod(M_n)$. Any automorphism of $(M_n,J)$ fixes $c_1(M_n,J)$, hence its image in $\Mod(M_n)$ is conjugate into the stabilizer $W_n$ of $K_{M_n}$.

Using the language of Carter graphs to enumerate the conjugacy classes of $W_n$, the following proposition addresses the complex Nielsen realization problem for irreducible elements of $\Mod^+(M_n)$ for $n = 3, 4, 6$.

\begin{proposition} \label{3_4_6_complex_realizable}
    Let $n=3,4,6$ and let $w \in W_n$ be irreducible. The class $w$ is realizable by a complex automorphism $\varphi \in \Aut(M_n, J)$ of some del Pezzo surface $(M_n, J)$. 
\end{proposition}
\begin{proof}
If $n = 3$ then there is a unique complex structure $J$ turning $(M_3, J)$ into a del Pezzo surface (up to isomorphism) and the natural map $\Aut(M_3, J) \to W_3$ is surjective and admits a section by \cite[Theorem 8.4.2]{dolgachev}. Therefore, any $w \in W_3$ is realizable by a complex automorphism of a del Pezzo surface $(M_3, J)$.

Suppose that $n = 4$ or $6$. By \Cref{4and6iff}, $w$ generates a cuspidal conjugacy class of $W_n$. \cite[Example 3.1.16, and Table B.4, p.407]{geck2000characters} lists the cuspidal conjugacy classes of $W_n$ by their Carter graphs. Note that each conjugacy class of $W_6$ listed in \cite[Table B.4, p.407]{geck2000characters} appears in \cite[Table 9]{dolgachev--iskovskikh}, which enumerates conjugacy classes of $W_n$ that are realizable by an automorphism of a del Pezzo surface diffeomorphic to $M_n$. 

There is only one cuspidal class in $W_4$, represented by Coxeter elements (cf. Section \ref{sec:coxeter-irred-real}). One can write down explicit birational transformations of $\CP^2$ that lift to automorphisms of a complex structure $(M_4,J)$ realizing the Coxeter elements; see \cite[Theorem 11.1]{mcmullen}. Since the fixed space in $H_2(M_4, \Z)$ of a Coxeter element is $\Z\{K_4\}$, \Cref{irred-minimal} and \cite[Theorem 3.8]{dolgachev--iskovskikh} imply that $(M_4,J)$ is a del Pezzo surface.
\end{proof}

The following corollary considers the complex realizability of the irreducible conjugacy classes of $W_n$. 
\begin{corollary}[{cf. \cite[Table 9]{dolgachev--iskovskikh}}, {\cite[Tables 5, 9, 10]{carter}}]\label{CarterClassify}
    Let $3 \leq n \leq 7$ and let $w \in W_n$ be irreducible. Suppose that $H_2(M_n, \Z)^{\langle w \rangle} \cong \Z$. Let $\Gamma_w$ denote the Carter graph of the conjugacy class of $w$. If $w$ is not realizable by a complex automorphism of any del Pezzo surface $(M_n, J)$ then 
    \begin{enumerate}[label = (\alph*)]
        \item $n = 5$ and $\Gamma_w$ is of type $D_5(a_1)$ or $D_2 + D_3$, or
        \item $n = 7$ and $\Gamma_w$ is of type $A_7$, $D_4 +3A_1$, $D_6 +A_1$, or $E_7(a_3)$.
    \end{enumerate}
\end{corollary}
\begin{proof}
Suppose that $w \in W_n$ is irreducible with $H_2(M_n, \Z)^{\langle w \rangle} \cong \Z$ and that $w$ is not realizable by a complex automorphism of a del Pezzo surface $(M_n, J)$. \Cref{irredclassification} shows that up to conjugacy in $\Mod^+(M_n)$, the class $w$ generates a cuspidal conjugacy class in $W_n$. If $n = 3, 4$ or $6$ then \Cref{3_4_6_complex_realizable} shows that all irreducible classes in $\Mod^+(M_n)$ are realizable as automorphisms of a del Pezzo surface $(M_n,J)$. Therefore, we need only consider cuspidal conjugacy classes in $W_5$ and $W_7$, which are enumerated in \cite[Table 5]{carter} (restricted to signed cycle-types described in Lemma \ref{negativeEven}) and \cite[Table B5]{geck2000characters} respectively. 

In the case of $n = 5$, there are three cuspidal conjugacy classes of $W_n$, with Carter graphs $D_5$, $D_5(a_1)$, and $D_2 +D_3$. The Carter graph $D_5$ corresponds to the conjugacy class of Coxeter elements of $W_5 \cong W(D_5)$ (cf. Section \ref{sec:coxeter-irred-real}). See \cite[Theorem 11.1]{mcmullen} for explicit birational transformations of $\CP^2$ that lift to automorphisms of a complex structure $(M_5,J)$ realizing the Coxeter elements of $W_5$. Since the fixed space in $H_2(M_5, \Z)$ of a Coxeter element is $\Z\{K_5\}$, \Cref{irred-minimal} and \cite[Theorem 3.8]{dolgachev--iskovskikh} imply that $(M_5,J)$ is a del Pezzo surface. Because $w$ is not realizable by a complex automorphism of a del Pezzo surface $(M_5, J)$, the Carter graph $\Gamma_w$ is not of type $D_5$ and must be of type $D_5(a_1)$ or $D_2 +D_3$. 

In the case of $n = 7$, any class appearing in \cite[Table 9]{dolgachev--iskovskikh} (cf. \cite[Table 1]{blanc}) is realizable by an automorphism of a del Pezzo surface $(M_n, J)$. The conjugacy classes listed in the statement of the corollary are exactly the cuspidal conjugacy classes of $W_7$ (from \cite[Table B5]{geck2000characters}) that do not appear in \cite[Table 9]{dolgachev--iskovskikh}.
\end{proof}

\begin{remark}\label{rmk:mcmullen}
In some special cases, there are viable, non-enumerative approaches to complex realization. For example, McMullen \cite[Theorem 7.2]{mcmullen} studies the action of the Weyl group $W_n$ on $\PP(\Z^{1,n} \otimes \C)$ and shows that if $w \in W_n$ fixes a point in $\PP(\Z^{1,n} \otimes \C)$ that pairs nontrivially with every root $\alpha \in \mathbb E_n$ then $w$ is realizable by a complex automorphism of some blowup $\Bl_{p_1, \dots, p_n} \CP^2$ diffeomorphic to $M_n$. While McMullen applies this to Coxeter elements of $W_n$ for all $n \neq 9$, a similar linear algebra check shows that his work also implies that irreducible elements of $\Mod^+(M_n)$ with $n \leq 8$ of odd, prime order are realizable by complex automorphisms. However, the realizability of these elements turns out to be subsumed by \cite{mcmullen}; see Theorem \ref{thm:primes-are-powers}. Regardless, it would be interesting if McMullen's work \cite{mcmullen} or other moduli-theoretic tools could be used to give non-enumerative proofs of Nielsen realization for finite subgroups of $\Mod(M_n)$ in general.
\end{remark}

\section{Comparing metric and complex Nielsen realization}\label{sectionMetComplex}

In this section we study the metric Nielsen realization problem on $M_n$ and compare it to the complex Nielsen realization problem. The natural distinguished class of metrics on the del Pezzo surfaces $M$ are Einstein metrics, and for $M = \CC\PP^1 \times \CC\PP^1$ or $\Bl_{p_1, \dots, p_n} \CP^2$ where $n =0$ or $3\leq n \leq 8$, K\"{a}hler--Einstein metrics.

\begin{definition}
A metric $g$ on $M$ is an \textit{Einstein metric} for a constant $\lambda$ if its Ricci curvature tensor satisfies $\Ric(g) = \lambda g$. It is \textit{K\"{a}hler--Einstein} if $g$ is additionally a K\"{a}hler metric on $(M,J)$ for some complex structure $J$.
\end{definition}

By work of Tian \cite[Main Theorem]{tian1990calabi}, the del Pezzo surfaces $\CP^2$, $\CP^1 \times \CP^1$, and $\Bl_{p_1, \dots, p_n} \CP^2$ with $3 \leq n \leq 8$ are precisely the compact complex surfaces that admit K\"ahler--Einstein metrics with $\lambda > 0$. There also exist conformally K\"ahler, Hermitian--Einstein metrics (the Page and Chen--LeBrun--Weber metrics) on the del Pezzo surfaces $\Bl_{p_1}\CP^2$ and $\Bl_{p_1, p_2} \CP^2$ respectively (see \cite{page} and \cite{chenlebrunweber}; cf. \cite[Proposition 2]{lebrunComplex}).

Let $M$ be a smooth manifold underlying any del Pezzo surface with a Riemannian metric $h$. Let $\mathcal{H}_+^2(M,h)$  denote the space of cohomology classes of self-dual, harmonic $2$-forms on $(M, h)$. Then $\dim_{\R}(\mathcal{H}_+^2(M,h)) = 1$ by Hodge theory because $b_+^2(M) = 1$. Any isometry $\varphi \in \Isom(M,h)$ preserves $\mathcal{H}_+^2(M,h)$. Following LeBrun \cite{lebrun2015einstein}, we say that the conformal class $[h]$ of such a Riemannian metric $h$ is of \textit{positive symplectic type} if $W^+(\omega, \omega) > 0$ everywhere for $\omega$ a self-dual harmonic form for $h$ and $W^+$ the self-dual Weyl tensor of $h$. In this case, $\omega$ is a \emph{symplectic} form on $M$ because $\omega$ is nowhere zero and is self-dual. We say that the metric $h$ is of positive symplectic type if its conformal class $[h]$ is of positive symplectic type. 

According to LeBrun \cite[Theorem A]{lebrun2015einstein}, any Einstein metric of positive symplectic type on a del Pezzo manifold is isometric to a K\"ahler--Einstein metric with $\lambda > 0$, (a constant multiple of) the Page metric, or (a constant multiple of) the Chen--LeBrun--Weber metric, and conversely any such metric is of positive symplectic type. 

\subsection{The blowups at more than 2 points}\label{blowup2more}

In this section we deduce the equivalence of the metric and complex Nielsen realization problems from results of LeBrun and Bando--Mabuchi. 
\begin{lemma}[{Bando--Mabuchi \cite{bando1987uniqueness}, LeBrun \cite{lebrun2015einstein}}]\label{MCequiv}
Let $3 \leq n \leq 8$ and let $G \subset \Diff^+(M_n)$ be a finite subgroup. The following are equivalent:
\begin{enumerate}[label=(\alph*)]
    \item There exists an Einstein metric $g$ of positive symplectic type on $M_n$ so that $G \subset \Isom(M_n, g)$ and the induced action of $G$ on $\mathcal{H}_+^2(M_n,g) \subset H_2(M_n,\RR)$ preserves the orientation of this line.
    \item There exists a complex structure $J$ on $M_n$ so that $(M_n, J)$ is a del Pezzo surface and $G \subset \Aut(M_n,J)$.
    \item There exists a K\"ahler--Einstein pair $(g, J)$ on $M_n$ so $G \subset \Aut(M_n, g, J)$ on a del Pezzo surface $(M_n,J)$.
\end{enumerate}
\end{lemma}
\begin{proof}
((c) $\Rightarrow$ (a)) If $G \subset \Aut(M,g,J)$, then $G \subset \Isom(M,g)$. By \cite[Theorem A]{lebrun2015einstein}, the Einstein metric $g$ is in fact of positive symplectic type. Moreover, $G$ preserves both $g$ and $J$, meaning it preserves the K\"ahler form $\omega$. Because $\omega$ is self-dual and harmonic, $G$ preserves the orientation of the line $\mathcal H_+^2(M_n, g).$

((c) $\Rightarrow$ (b)) This is clear from the definitions: if $ G \subset \Aut(M,g,J)$, then $G \subset \Aut(M,J)$.

((b) $\Rightarrow$ (c)) Let $ G \subset \Aut(M,J)$ for $J$ a complex structure making $(M,J)$ a del Pezzo surface. By the existence result of \cite{tian1990calabi} and \cite[Theorem C]{bando1987uniqueness}, there exists an $G$-invariant Einstein metric $g$ that makes $(M,g,J)$ a K\"{a}hler--Einstein del Pezzo surface. 

((a) $\Rightarrow$ (c)) Let $\varphi \in G \subset \Isom(M_n ,g)$ for $g$ an Einstein metric of positive symplectic type, so that $\varphi$ preserves the orientation of the line $\mathcal H_{+}^2(M_n, g)$. By \cite[Theorem A]{lebrun2015einstein}, there exists a complex structure $J$ on $M_n$ so that $(M_n, g, J)$ is a K\"ahler--Einstein pair with $\lambda > 0$. Let $\omega(u,v) = g(Ju,v)$ be the associated symplectic form, where $\omega$ is self-dual and harmonic with respect to $g$. Since $\varphi$ preserves $\mathcal{H}^2_+(M_n,g)$ along with its orientation and since $\varphi$ has finite order, we have $\varphi^*\omega = \omega$ by the uniqueness of harmonic representatives of cohomology classes. Therefore, $\varphi$ is contained in $\mathrm{Isom}(M_n,g) \cap \mathrm{Symp}(M_n,\omega)$. The $2$-out-of-$3$ rule then implies $\varphi \in \mathrm{Aut}(M_n,J)$.
\end{proof}

\begin{remark}[On high degree del Pezzo surfaces]
    One can directly calculate that the full mapping class group $\Mod(M)$ is Einstein metric realizable for $M = \CP^2$ and $\CP^1 \times \CP^1$; the smooth representatives coming from complex conjugation and complex automorphisms in \cite{lee} preserve the Fubini--Study metric. 

    Although $M_1$ and $M_2$ do not admit K\"ahler--Einstein metrics, they admit distinguished conformally-K\"ahler, Einstein metrics called the Page metric and Chen--LeBrun--Weber metric respectively. These metrics are the unique Hermitian--Einstein metrics on $M_1$ and $M_2$ \cite[Theorem A]{lebrun2012einstein} and the unique Einstein metrics of positive symplectic type \cite[Theorem A]{lebrun2015einstein}, similarly to the Kahler--Einstein metrics. It would be interesting to analyze the metric realizability of finite-order elements of $\Mod(M_1)$ or $\Mod(M_2)$, e.g. the class $-I \in\Mod(M_n)$ for $n = 1, 2$, which is smoothly represented by complex conjugation.
\end{remark}

\subsection{Metric non-realizability by a complex automorphism}

The results of Tian and LeBrun on Einstein metrics with $\lambda >0$ do not apply to a general rational surface that is not a del Pezzo surface. The following proposition makes use of other complex structures on $M_n$ to produce a mapping class that is realizable by a complex automorphism but not by an isometry of an Einstein metric of positive symplectic type.
\begin{proposition}\label{Jonquiere--not--einstein}
There exists an irreducible mapping class $f \in \MOD(M_n)$ of finite order that is realizable by a complex automorphism $\varphi \in \Aut(M_7, J)$ but is not realizable by an isometry of any Einstein metric of positive symplectic type on $M_7$.
\end{proposition}
\begin{proof}
By \cite[Proposition 3.14]{lee-involutions}, there is an irreducible, order-$2$ mapping class $f \in \Mod^+(M_7)$ that is realizable by a de Jonqui\'{e}res involution $\varphi \in \Aut(M_7, J)$ of algebraic degree 4 for some complex structure $J$ on $M_7$. By \cite[Lemma 3.7]{lee-involutions}, $H_2(M_7; \Q)$ decomposes as a sum of the eigenspaces
\[
    H_2(M_7; \QQ) \cong V_+ \oplus V_-
\]
with $\dim_\QQ(V_+) = 2$ and $\dim_\QQ(V_-) = 6$, where $V_+$ and $V_-$ denote the $1$-eigenspace and $(-1)$-eigenspace of $\varphi$ respectively. Since $f$ fixes the canonical class $K_{M_7}$, the trace of $f$ on $\mathbb E_7 = \ZZ\{K_{M_7}\}^\perp$ is $1 - 6 = -5$. Finally, the classification of automorphisms on degree-$2$ del Pezzo surfaces \cite[Table 5]{dolgachev--iskovskikh} shows that no such automorphism of order $2$ has trace $-5$ on $\EE_7$. In other words, $f$ is not realizable by complex automorphisms of del Pezzo surfaces. 

Suppose that $f$ is represented by an order-$2$ isometry $\psi \in \Isom(M_7, h)$ for some Einstein metric $h$ of positive symplectic type. Since $\psi$ is an isometry, it preserves the line $\mathcal{H}_h^2(M_7) \subset H^2(M_7,\RR)$ of self-dual harmonic 2-forms of $(M_7, h)$. Because $\mathrm{PD}(K_{M_7})$ is fixed by $f$ and has positive square, the $(-1)$-eigenspace $\mathrm{PD}(V_-)$ is a negative definite subspace of $H^2(M_7, \RR)$. Therefore, $f$ preserves the orientation of the line $\mathcal H^2_h(M_7)$.

\Cref{MCequiv} now implies that $\psi$ is an automorphism of some del Pezzo surface $(M_7, J)$, which yields a contradiction. The order-$2$ mapping class $f$ is not realizable by any isometry of an Einstein metric of positive symplectic type.
\end{proof}

\begin{remark}
Compare with \cite[Theorem 1.4]{farb--looijenga}, which finds a subgroup $G \leq \Mod(M)$ isomorphic to the alternating group $A_4$ that can be realizable by Ricci-flat isometries but not by complex automorphisms.
\end{remark}

\section{Comparing complex and smooth Nielsen realization}\label{ComplexVSSmooth}

The goal of this section is to show that the conjugacy classes of $W_n$ listed in \Cref{CarterClassify} are not realizable by diffeomorphisms of the same order. The casework of this section forms the bulk of the proof of \Cref{mainthm1}. The main tools come from the theory of finite group actions on $4$-manifolds, e.g. the $G$-signature theorem, Edmonds' theorem, the Riemann--Hurwitz formula.

\subsection{Realization obstruction lemma}

This section presents a lemma that provides a homological criterion for obstructing smooth Nielsen realization for $4$-manifolds. We begin by defining relevant terminology, following Gordon \cite{gordon1986g}. Let $M$ be a closed, oriented, smooth 4-manifold and let $G  = \langle \varphi \rangle \subset \Diff^+(M)$ be a finite subgroup of order $m$. Because $\varphi_*$ preserves the intersection form $Q_M$ on $H_2(M,\ZZ)$, it also preserves the induced Hermitian form 
\begin{align*}
    \Phi : H_2(M,\CC) \times H_2(M,\CC) \to \CC \\
    (\alpha c_1, \beta c_2) \mapsto (\alpha\Bar{\beta})Q_M(c_1,c_2),
\end{align*}
for any $\alpha, \beta\in \CC$ and $c_1, c_2 \in H_2(M,\ZZ)$. There is a $G$-invariant orthogonal direct sum decomposition
$$H_2(M,\CC) = H^+ \oplus H^- \oplus H^0$$
where $\Phi$ is positive- and negative-definite on $H^+$ and $H^-$ respectively, and zero on $H^0$.
\begin{definition}[{\cite[p. 162]{gordon1986g}}]
    The \textit{$\varphi$-signature} is defined to be
    $$\mathrm{sign}(\varphi,M) = \mathrm{Tr}(\varphi_*|_{H^+}) - \mathrm{Tr}(\varphi_*|_{H^-})$$
\end{definition}

Let $\Fix(\varphi)$ denote the pointwise fixed set of $\varphi$ acting on $M$. The set $\Fix(\varphi)$ consists of a finite union of isolated points and disjoint, closed, connected 2-manifolds, and $\Fix(\varphi)$ is orientable if $m > 2$ \cite[Proof of Lemma 3.5(3)]{farb--looijenga}. For each isolated fixed point $p \in \Fix(\varphi)$, there exists a normal neighborhood around $p$ in $M$ that is $G$-equivariantly diffeomorphic to $(\theta_1, \DD^2) \times (\theta_2, \DD^2)$, where $\varphi$ acts on $\DD^2$ by a rotation by $\theta_k \in \frac{2 \pi}{m}\ZZ$, for $k=1,2$. Similarly, for each connected surface $F \subset \Fix(\varphi)$, there exists a normal neighborhood of $F$ in $M$ that is $G$-equivariantly diffeomorphic to $(\psi, E)$, where $E$ is a $\DD^2$-bundle over $F$, and $\varphi$ acts on each fiber of $E\to F$ as a rotation by $\psi \in \frac{2\pi}{m}\ZZ$.

If $F$ is orientable, let $e(F):= Q_{M} ([F],[F])$. Otherwise, see \cite{gordon1986g} for the definition.
\begin{theorem}[{$G$-signature theorem \cite[Theorem 2]{gordon1986g}; cf. \cite[4.1(2), Theorem 9.1.1]{hirzebruch--zagier}}]\label{gordonmain}
    With the above notation,
    $$\mathrm{sign}(\varphi,M) = - \sum_{p} \cot\left(\frac{\theta_1}{2}\right)\cot\left(\frac{\theta_2}{2}\right) + \sum_F e(F)\csc^2\left(\frac{\psi}{2}\right)$$
\end{theorem}

We point out one special case which is used in the proof of \Cref{gordonmain}.

\begin{corollary}[{\cite[Lemma 7]{gordon1986g}}]\label{gordoncor}
    If $\varphi$ does not have any fixed points in $M$ then 
    $$\mathrm{sign}(\varphi,M) = 0.$$
\end{corollary}

We now state Edmonds' theorem on fixed sets of prime cyclic actions on simply-connected 4-manifolds. 
\begin{proposition}[{Edmonds \cite[Proposition 2.4]{edmonds}}]\label{edmonds}
    Let $M$ be a closed, oriented, simply-connected $4$--manifold, and let $G  = \langle \varphi \rangle \subset \Homeo^+(M)$ be a finite cyclic group of prime order $p$. Suppose that $\Fix(\varphi) \neq \emptyset$. Let $t$, $c$
and $r$ denote the number of trivial, cyclotomic and regular summands of $H^2(M,\ZZ)$ as a $\ZZ[G]$-module. Then
$$b_0(\Fix(\varphi), \FF_p) + b_2(\Fix(\varphi), \FF_p) = t + 2$$
and
$$b_1(\Fix(\varphi) , \FF_p) = c,$$
where $b_i(\Fix(\varphi), \FF_p) = \dim_{\FF_p}( H^i(\Fix(\varphi), \FF_p))$ denotes the $i$-th Betti number with $\FF_p$-coefficients.
\end{proposition}

Finally, consider the \emph{Lefschetz number} $\Lambda(f)$ of a mapping class $f \in \Mod(M)$, which for a simply-connected $4$-manifold $M$ is
\[
    \Lambda(f) = 2 + \mathrm{Tr}(f: H_2(M, \ZZ) \to H_2(M, \ZZ)).
\]
In the case of finite-order diffeomorphisms $\varphi$ of $4$-manifolds, a version of the Lefschetz fixed point theorem computes the Euler characteristic $\chi(\Fix(\varphi))$ of $\Fix(\varphi)$. 
\begin{theorem}[{cf. Kwasik--Schultz \cite[Theorem 1]{kwasik--schultz}, \cite[Section 5.1]{Edmonds--locally--linear}}]\label{thm:lefschetz-fixed}
Let $M^4$ be a closed, smooth $4$-manifold and let $\varphi \in \Diff^+(M)$ have finite order. Then the Euler characteristic $\chi(\Fix(\varphi))$ of $\Fix(\varphi)$ is equal to the Lefschetz number $\Lambda([\varphi])$ of $[\varphi]$.
\end{theorem}

Now we can state a lemma that obstructs smooth realization for a finite order mapping class using homological data. 
\begin{lemma}\label{ObstructionLemma}
Let $M$ be a closed, oriented, smooth, simply-connected $4$-manifold and let $f\in \MOD(M)$ be an element of order $m$. Suppose that for some prime $p$ that divides $m$, the $\ZZ[\langle f^{m/p}\rangle]$-module $H_2(M, \ZZ)$ has no cyclotomic summands. Suppose further that $f$ satisfies one of the following conditions:
\begin{enumerate}[label=(\alph*)]
    \item $\Lambda(f) = 0$ and $\mathrm{sign}(f,M) \neq 0$, or
    \item $\Lambda(f) < 0$.
\end{enumerate}
Then there does not exist any diffeomorphism $\varphi \in \Diff^+(M)$ of finite order such that $[\varphi] = f$.
\end{lemma}
\begin{proof}
    Assume for the sake of contradiction that there exists a diffeomorphism $\varphi \in \Diff^+(M)$ of order $k$ satisfying $[\varphi] = f$. Then $m$ and $p$ divide $k$, and $\varphi^{k/p}$ is a diffeomorphism of order $p$ representing $f^{k/p}$. 

    We claim that $H_2(M, \Z)$ has no cyclotomic summands as a $\Z[\langle \varphi^{k/p}\rangle]$-module. Because $(f^{k/p})^p = 1$, the class $f^{k/p}$ has order $1$ or order $p$. If $f^{k/p} = 1$ then $H_2(M, \ZZ)$ has no cyclotomic summands as a $\ZZ[\langle \varphi^{k/p}\rangle]$-module. If $f^{k/p}$ has order $p$ then $\langle f^{k/p} \rangle$ = $\langle f^{m/p} \rangle$, the unique subgroup of order $p$ in $\langle f \rangle$. Because $H_2(M, \ZZ)$ has no cyclotomic summands as a $\ZZ[\langle f^{m/p} \rangle]$-module, it also has no cyclotomic summands as a $\ZZ[\langle \varphi^{k/p} \rangle]$-module. 
    
    Condition (a) implies that $\Fix(\varphi) \neq \emptyset$ by \Cref{gordoncor}, and Condition (b) implies that $\Fix(\varphi) \neq \emptyset$ by \Cref{thm:lefschetz-fixed}. In either case, \Cref{edmonds} shows that
    \[
        b_1(\Fix(\varphi^{k/p}) , \FF_{p}) = 0.
    \]
    By the classification of surfaces, the $2$-dimensional components of $\Fix(\varphi^{k/p})$ are spheres. Thus $\Fix(\varphi^{k/p})$ is a disjoint union of spheres and points.
    
    \Cref{thm:lefschetz-fixed} further implies that $\chi(\Fix(\varphi)) = \Lambda(f) \leq 0$. However, $\Fix(\varphi)\subset \Fix(\varphi^{k/p})$ is a disjoint, nonempty union of spheres and points, which contradicts the inequality $\chi(\Fix(\varphi))\leq 0$. Therefore no such diffeomorphism $\varphi$ exists.
\end{proof}

\begin{remark} \label{obstruct_locally_linear}
    The proof of Lemma \ref{ObstructionLemma} also obstructs the Nielsen realization problem for locally linear actions on $M$ as a topological manifold: see \cite[Chapter 4]{bredon}, \cite{Chen--survey} for the defintion of a locally linear map. A generalization of an argument of Wall implies that the $G$--signature theorem holds for locally linear actions on topological $4$--manifolds \cite[Remark on p.709]{Wilczynski}.
\end{remark}

The following lemma records the signature of irreducible mapping classes on $M_n$.
\begin{lemma}\label{WeylSign}
    Let $\varphi \in \Diff^+(M_n)$ be a diffeomorphism of finite order such that $[\varphi]$ is an element of $W_n \subset \Mod(M_n)$ . Then
    $$\mathrm{sign}(\varphi,M_n) = 1 - \mathrm{Tr}(\varphi_* |_{\EE_n}).$$
\end{lemma}
\begin{proof}
    Because $\varphi_*(K_{M_n}) = K_{M_n}$, the automorphism $\varphi_*$ preserves the orthogonal direct sum decomposition
    $$H_2(M_n,\CC) = \CC\{K_{M_n}\}\oplus (\EE_n \otimes \CC).$$
    Note that $\Phi$ is positive-definite on the first summand because $Q_{M_n}(K_{M_n}, K_{M_n}) > 0$. Analogously, $\Phi$ is negative definite on $\EE_n \otimes \CC$ because $Q_{M_n}$ is negative-definite on $\EE_n$.
\end{proof}

\subsection{Nonrealizability in $M_5$}

According to the proof of \cite[Proposition 8.6.7]{dolgachev}, there is an isomorphism $W_5 \cong W(D_5)$ via the quotient $W_5 \to S_5$ given by the action of $W_5$ on the set of unordered pairs
\[
\{\{H-E_k, \, 2H - E_1 - E_2 - E_3 - E_4 - E_5 + E_k\} : 1 \leq k \leq 5\}. 
\]
Throughout this section, we use the signed cycle-type notation (cf. \Cref{defn:cycle-type}) to study the conjugacy classes of $W_5$.
\begin{lemma}\label{lem:d2d3-square}
    Let $f \in W_5$ have Carter graph $D_2 + D_3$. Then $G = \langle f^2 \rangle \cong \ZZ/2\ZZ$ and there is an isomorphism of $\ZZ[G]$-modules
    $$H_2(M_5,\ZZ) \cong \ZZ^{\oplus 2} \oplus \ZZ[G]^{\oplus 2}.$$
\end{lemma}
\begin{proof}
    According to \cite[Table 5]{carter}, the signed cycle-type of $f$ is $[\bar{2}\bar{1}\bar{1}\bar{1}]$ and so the signed cycle-type of $f^2$ is $[\bar{1} \bar{1}111]$. This signed cycle-type is achieved by the automorphism $\RRef_{E_1-E_2} \circ \RRef_{H-E_3-E_4-E_5}$, so up to conjugacy in $W_5$,
    \[
        f^2 = \RRef_{E_1-E_2} \circ \RRef_{H-E_3-E_4-E_5}.
    \]
    So up to conjugacy in $W_5$, $f^2$ acts by the identity on the first summand and by swapping the generators in each of the latter two summands below:
    $$H_2(M_5,\ZZ) = \ZZ\{H-E_4 , H - E_5\}\oplus \ZZ\{H - E_4 - E_5, E_3\} \oplus \ZZ\{E_1 , E_2\}. \qedhere$$
\end{proof}

The following two propositions conclude the nonrealizability proofs for $M_5$. 
\begin{proposition}\label{D2D3}
    Let $f \in W_5$ have Carter graph $D_2 + D_3$. Then $f$ is not realizable by any diffeomorphism $\varphi \in \Diff^+(M_5)$ of finite order.
\end{proposition}
\begin{proof}
    According to \cite[Table 3]{carter}, the characteristic polynomial $\chi_f(t)$ of $f|_{\EE_5}$ is 
    $$\chi_f(t) = (t^2+1)(t+1)^3.$$
    Therefore $\mathrm{Tr}(f|_{\EE_5}) = -3$, and
    $$\Lambda(f) = 1 + (1-3) +1= 0, \qquad \mathrm{sign}(f, M_5) = 1 - \Tr(f|_{\EE_5}) =4$$
    by \Cref{WeylSign}. By \Cref{lem:d2d3-square}, $H_2(M_5, \ZZ)$ does not have any cyclotomic summands as a $\ZZ[\langle f^2 \rangle]$-module. Because $\Lambda(f) = 0$ and $\mathrm{sign}(f, M_5) \neq 0$, applying \Cref{ObstructionLemma} with $p = 2$ shows that there does not exist any finite-order diffeomorphism $\varphi$ realizing $f$. 
\end{proof}

\begin{proposition}\label{D5a1}
    Let $f\in W_5$ have Carter graph $D_5(a_1)$. Then $f$ is not realizable by any diffeomorphism $\varphi \in \Diff^+(M_5)$ of finite order.
\end{proposition}
\begin{proof}
    The signed cycle-type of $f$ is $[\bar{3}\bar{2}]$ according to \cite[Table 5]{carter}, and so the signed cycle-type of $f^3$ is $[\Bar{2}\Bar{1}\Bar{1}\Bar{1}]$. Therefore $f^3$ has Carter graph $D_2 + D_3$ and $f^3$ is not realizable by any diffeomorphism of finite order by \Cref{D2D3}.
\end{proof}

\subsection{Nonrealizability of $D_4 + 3A_1$, $D_6 + A_1$, and $E_7(a_3)$ on $M_7$}
In this section we address the nonrealizability of three conjugacy classes of $W_7 \subset \Mod(M_7)$. The first proposition handles the conjugacy class of type $D_4 + 3A_1$. 
\begin{proposition}\label{D43A1nonrealize}
    Let $f \in W_7$ have Carter graph $D_4 + 3A_1$. Then $f$ is not realizable by any diffeomorphism $\varphi \in \Diff^+(M_7)$ of finite order.
\end{proposition}
\begin{proof}
    Consider the following element of $W_7$:
    \[
        w = (\RRef_{H-E_1-E_2-E_3}\circ\RRef_{E_2-E_3}\circ\RRef_{E_4-E_5}\circ\RRef_{E_6-E_7})\circ(\RRef_{H-E_1-E_4-E_5}\circ\RRef_{H-E_1-E_6-E_7}\circ\RRef_{E_1-E_3}).
    \]
    With respect to the $\Z$-basis $\{H, E_1, \dots, E_7\}$ of $H_2(M_7, \Z)$, the matrix forms of $w$ and $w^2$ are
    \[
        w = \scalebox{0.75}{$\begin{pmatrix}
         4 & 1 & 1 & 3 & 1 & 1 & 1 & 1  \\
        -3 & -1& -1& -2& -1& -1& -1& -1 \\
        -1 & 0& -1& -1&  0&  0&  0&  0 \\
        -1& -1&  0& -1&  0&  0&  0&  0 \\
        -1 &  0&  0& -1& -1&  0&  0&  0 \\
        -1 &  0&  0& -1&  0& -1&  0&  0\\ 
        -1& 0&  0& -1&  0&  0& -1&  0 \\
        -1 & 0&  0& -1&  0&  0&  0& -1 
        \end{pmatrix}$}, \qquad w^2 = \scalebox{0.75}{$\begin{pmatrix}
        5 & 0 & 2 & 2 & 2 & 2 & 2 & 2 \\
        -2 & 0 & -1 & 0& -1& -1& -1& -1\\
        -2& 0 & 0 &-1& -1& -1& -1& -1 \\
        0& 1 & 0 & 0 & 0 & 0 & 0 & 0 \\
        -2 & 0& -1& -1&  0& -1& -1& -1 \\
        -2& 0& -1& -1& -1&  0& -1& -1 \\
        -2& 0& -1& -1& -1& -1&  0& -1 \\
        -2& 0 &-1& -1& -1& -1& -1 & 0 
        \end{pmatrix}$}.
    \]
    According to \cite[Table 3]{carter}, the characteristic polynomial $\chi_f (t)$ of $f|_{\EE_7}$ is 
    \[
    \chi_f (t) = (t^2 - t+1) (t+1)^5,
    \]
    and one can compute that the characteristic polynomial $\chi_w(t)$ of $w|_{\EE_7}$ equals $\chi_f(t)$. By \cite[Lemma 3.1.10]{geck2000characters}, $w$ generates a cuspidal conjugacy class of $W_7$. According to \cite[Table B.5]{geck2000characters}, there are four cuspidal conjugacy classes of $W_7$ of order $6$, with Carter graphs
    \[
        E_7(a_4), \quad D_6(a_2) + A_1, \quad A_5 + A_2, \quad D_4 + 3A_1.
    \]
    According to \cite[Table 3]{carter}, the characteristic polynomials of $E_7(a_4)$, $D_6(a_2) + A_1$, and $A_5 + A_2$ acting on $\mathbb E_7$ are
    \[
        (t^2-t+1)^3(t+1), \quad (t^3+1)^2(t+1), \quad (t^5 + t^4+t^3+t^2+t+1)(t^2+t+1).
    \]
    respectively. Therefore, $w$ and $f$ must determine the same conjugacy class of $W_7$, namely that of type $D_4 + 3A_1$. Therefore after possibly conjugating $f$ by an element of $W_7$, we may assume that $f = w$. 

    Let $f_3 := f^2$. Considering the matrix form of $f_3$ shows that $f_3$ preserves the following subgroup
    \[
        \Z\{E_1, \, E_3, \, 2H - E_2-E_4-E_5-E_6-E_7\} \leq H_2(M_7, \Z),
    \]
    and that this subgroup is isomorphic to the regular representation of $\langle f_3 \rangle \cong \Z/3\Z$. Moreover, the restriction of $Q_{M_7}$ to $\Z\{E_1, \, E_3, \, 2H - E_2-E_4-E_5-E_6-E_7\}$ is unimodular, and hence there is an orthogonal direct sum decomposition
    \[
        H_2(M_7, \Z) \cong \Z\{E_1, \, E_3, \, 2H - E_2-E_4-E_5-E_6-E_7\}  \oplus \Z\{E_1, \, E_3, \, 2H - E_2-E_4-E_5-E_6-E_7\}^\perp.
    \]
    Compute that the characteristic polynomial $\chi_{f_3}(t)$ of $f_3|_{\EE_7}$ is
    \[
        \chi_{f_3}(t) = (t^2 + t + 1)(t-1)^5.
    \]
    By eigenvalue considerations, $f_3$ acts trivially on $\Z\{E_1, \, E_3, \, 2H - E_2-E_4-E_5-E_6-E_7\} ^\perp$. Therefore, there is an isomorphism of $\Z[\langle f_3 \rangle]$-modules
    \[
        H_2(M_7, \Z) \cong \Z[\langle f_3 \rangle] \oplus \Z^{5},
    \]
    where $\Z$ denotes the trivial $\langle f_3 \rangle$-representation. In other words, $H_2(M_7, \Z)$ has no cyclotomic summands as a $\Z[\langle f_3 \rangle]$-module. 

    Finally, note that $\Tr(f|_{\EE_7}) = -4$, and so the Lefschetz number $\Lambda(f) = 2 + (1-4) = -1$ is negative. Applying \Cref{ObstructionLemma} to $f$ with $p = 3$ shows that there does not exist any diffeomorphism $\varphi \in \Diff^+(M_7)$ of finite order with $[\varphi] = f$. 
\end{proof}

The next goal of this section is to show that certain conjugacy classes of order $10$ are not realizable by diffeomorphisms of finite order.

\begin{lemma}\label{D6A1fixedpoint}
    Let $f \in W_7$. Suppose that the characteristic polynomial $\chi_f(t)$ of $f|_{\mathbb E_7}$ is
    \[
        \chi_f(t) = (t^5 + 1)(t+1)^2.
    \]
    If there exists $\varphi \in \Diff^+(M_7)$ of order 10 such that $[\varphi] = f$ then $\Fix(\varphi) = \{p\}$ for some $p \in M_7$.
\end{lemma}
\begin{proof}
    The characteristic polynomial of $f_2 := f^5$ acting on $\EE_7$ is $\chi_{f_2}(t) = (t+1)^7$. Let $t_2$, $c_2$, and $r_2$ denote the number of trivial, cyclotomic, and regular summands of $H_2(M_7, \ZZ)$ as a $\ZZ[\langle f_2 \rangle]$-module respectively. By eigenvalue considerations, $t_2+r_2 = 1$, and so $c_2 + r_2 = 7$. In particular, $t_2 \leq 1$ and $c_2 \geq 6$.

    Suppose that there exists a diffeomorphism $\varphi \in \Diff^+(M_7)$ of order $10$ with $[\varphi] = f$. Because $\Tr(f|_{\EE_7}) = -2$, 
    \[
        \Lambda(f) = 1 +(1-2) +1 = 1.
    \]
    By \Cref{thm:lefschetz-fixed}, $\chi(\Fix(\varphi)) = 1$ and $\Fix(\varphi) \neq \emptyset$. 
    
    Let $\varphi_2:= \varphi^5$ and note that $\Fix(\varphi_2) \neq \emptyset$. Apply \Cref{edmonds} to see that
    $$b_0(\Fix(\varphi_2) , \FF_2) + b_2(\Fix(\varphi_2), \FF_2) \leq 3, \quad b_1(\Fix(\varphi_2),\FF_2) \geq 6,$$
    which implies that the $2$-dimensional part of $\Fix(\varphi_2)$ is a connected surface $\Sigma \neq S^2$.

    Let $f_5 := f^2$ be the mapping class of $\varphi_5 := \varphi^2$. The $2$--dimensional components of $\Fix(\varphi_5)$ are orientable because $\varphi_5$ has odd order, so the number $c_5$ of cyclotomic summands in $H_2(M_7, \ZZ)$ as a $\ZZ[\langle f_5 \rangle]$-module is even, by \Cref{edmonds}. The cyclotomic representation of $\Z/5\Z$ has rank $4$, and has no real eigenvalues, while $f_5$ has a $1$-eigenvector $K_{M_7}$ in $H_2(M_7, \ZZ)$, and so $c_5 \leq 1$. Altogether, we must have $b_1(\Fix(\varphi_5), \FF_5) = c_5 = 0$ by \Cref{edmonds}, and hence $\Fix(\varphi_5)$ does not contain $\Sigma$. 

    Since 
    \[
        \Fix(\varphi) \subset \Fix(\varphi_2) \cap \Fix(\varphi_5),
    \]
    $\Fix(\varphi)$ must consist of finitely many points. Because $\chi(\Fix(\varphi)) = 1$, we conclude that $\# \Fix(\varphi) = 1$. 
\end{proof}

Using \Cref{D6A1fixedpoint}, the next proposition handles nonrealization for the conjugacy class of type $D_6 + A_1$.
\begin{proposition}\label{D6A1nonrealize}
    Let $f \in W_7$. Suppose that the characteristic polynomial $\chi_f(t)$ of $f|_{\mathbb E_7}$ is
    \[
        \chi_f(t) = (t^5 + 1)(t+1)^2.
    \]
    Then $f$ is not realizable by any diffeomorphism $\varphi \in \Diff^+(M_7)$ of order $10$. In particular, if $f \in W_7$ has Carter graph $D_6 + A_1$ then $f$ is not realizable by any diffeomorphism of order $10$.
\end{proposition}
\begin{proof}
    Suppose for the sake of contradiction that there exists $\varphi \in \Diff^+(M_7)$ of order $10$ with $[\varphi] = f$. \Cref{D6A1fixedpoint} shows that $\Fix(\varphi) = \{p\}$. Let $G =\langle \varphi \rangle \cong \ZZ/10 \ZZ$ and suppose that a tubular neighborhood of $p$ is $G$-equivariantly diffeomorphic to $(\theta_1, \DD^2)\times (\theta_2, \DD^2)$. There exist $a, b\in \Z$ with $1 \leq a, b \leq 9$ and
    \[
        \theta_1 = \frac{2a\pi}{10}, \quad \theta_2 = \frac{2 b \pi}{10}.
    \]
    For any such choice of $a, b$, one can numerically compute that
    \[
        \cot\left(\frac{a\pi}{10}\right)\cot\left(\frac{b\pi}{10}\right) \neq -3 = - \mathrm{sign}(\varphi, M_7),
    \]
    where the last equality follows from \Cref{WeylSign} because $\mathrm{Tr}(\varphi_*|_{\mathbb E_7}) = -2$. This contradicts \Cref{gordonmain}.

    Now suppose that $f$ has Carter graph $D_6 + A_1$. By \cite[Table 3]{carter}, the characteristic polynomial $\chi_f(t)$ of $f|_{\mathbb E_7}$ is $\chi_f(t) = (t^5 + 1)(t+1)^2$, so the final claim follows.
\end{proof}

Finally, the following proposition handles nonrealizability for the conjugacy class of type $E_7(a_3)$.
\begin{proposition}\label{E7a3nonrealize}
    Let $f \in W_7$ have Carter graph $E_7(a_3)$. Then $f$ has order 30 and $f$ is not realizable by any diffeomorphism $\varphi \in \Diff^+(M_7)$ of order 30.
\end{proposition}
\begin{proof}
    By \cite[Table 3]{carter} and by eigenvalue considerations, the characteristic polynomials $\chi_f(t)$ and $\chi_{f^3}(t)$ of $f|_{\EE_7}$ and $f^3|_{\EE_7}$ respectively are 
    \[
        \chi_f(t) = (t^5 + 1)(t^2-t+1), \qquad \chi_{f^3}(t) = (t^5 + 1)(t+ 1)^2.
    \]
    By \Cref{D6A1nonrealize}, $f^3$ is not realizable by any diffeomorphism of order $10$.
\end{proof}

\subsection{Type $A_7$ on $M_7$}\label{sec:A7}
In this section we address the nonrealizability of the conjugacy class of $W_7$ of type $A_7$. Throughout, let $f \in W_7$ have Carter graph $A_7$ so that $f$ has order $8$ and suppose for the sake of contradiction that there exists a diffeomorphism $\varphi \in \Diff^+(M_7)$ of order $8$ with $[\varphi] = f$. Let
\[
    f_2 := f^4, \quad f_4 := f^2, \qquad \varphi_2 := \varphi^4, \quad \varphi_4 := \varphi^2
\]
so that each $f_k$ has order $k$ for $k = 2,4$ and each $\varphi_k$ has order $k$ and $[\varphi_k] = f_k$ for each $k = 2,4$.

The following lemma concerns the fixed sets of all powers of $\varphi$. 
\begin{lemma}\label{A7Lefschetz}
    For any $k \neq 0 \in \Z/8\Z$, the diffeomorphism $\varphi^k$ satisfies $\chi(\Fix(\varphi^k)) = 2$ and $\mathrm{sign}(\varphi^k, M_7) = 2$. 
\end{lemma}
\begin{proof}
    According to \cite[Table 3]{carter}, the characteristic polynomial $\chi_f(t)$ of $f|_{\EE_7}$ is 
    \begin{equation}\label{eqn:a7-characteristic}
        \chi_f(t) = t^7 + t^6 + t^5+t^4+t^3+t^2 + t + 1,
    \end{equation}
    so $\Tr(f|_{\EE_7}) = -1$. By eigenvalue considerations, compute also that $\Tr(f^k|_{\EE_7}) = -1$ for any $k \neq 0 \in \Z/8\Z$. By \Cref{thm:lefschetz-fixed} and \Cref{WeylSign},
    \[
        \chi(\Fix(\varphi^k)) = 1 + (1-1) +1 = 2, \qquad \mathrm{sign}(\varphi^k, M_7) = 1 - (-1) = 2. \qedhere
    \]
\end{proof}

In this lemma we constrain the possible fixed set of $\varphi_2$. The rest of this subsection is devoted to contradicting each of the following cases separately. 
\begin{lemma}\label{A7fixedlocus1}
    The fixed set $\Fix(\varphi_2)$ is diffeomorphic to one of:
    \[
    	X \sqcup \{p_1,p_2\} , \quad X \sqcup S^2, \text{ or } \quad \RR\PP^2 \sqcup \RR\PP^2,
    \]
    where $X$ is a connected surface with $\chi(X) = 0$.
\end{lemma}
\begin{proof}
Consider the element of $W_7$
\[
	w = \RRef_{2H - E_1-E_2-E_3-E_4-E_5-E_6} \circ (\RRef_{E_1-E_2} \circ \RRef_{E_2-E_3} \circ \RRef_{E_3-E_4} \circ \RRef_{E_4-E_5}  \circ \RRef_{E_5-E_6} \circ \RRef_{E_6-E_7}).
\]
The matrix forms of $w$ and $w^4$ with respect to the $\Z$-basis $\{H, E_1, \dots, E_7\}$ of $H_2(M_7; \Z)$ are
\[
    w =  \scalebox{0.75}{$\begin{pmatrix}
    5 & 2 & 2 & 2 & 2 & 2 &0 & 2\\
    -2 & -1 & -1& -1 & -1& -1&0& 0\\
    -2 & 0 & -1& -1 & -1& -1&0& -1\\
    -2 & -1 &0& -1 & -1& -1 &0& -1\\
    -2 & -1& -1& 0 & -1& -1&0& -1\\
    -2 & -1& -1& -1& 0& -1&0& -1\\
    -2 & -1& -1& -1& -1& 0&0& -1 \\
    0 & 0 &0 &0 & 0 &0&1 &0
    \end{pmatrix}$}, \qquad w^4 =  \scalebox{0.75}{$\begin{pmatrix}
    5& 2 & 2 & 0 & 2 & 2 & 2 & 2 \\
    -2 & -1 & -1 & 0 & -1 & 0 & -1 & -1 \\
    -2 & -1 & -1 & 0 & -1 & -1 & 0 & -1  \\
    0 & 0 & 0 & 1 & 0 & 0 & 0 & 0  \\
    -2 & -1 & -1 & 0 & -1 & -1 & -1 &  0 \\
    -2 & 0 & -1 & 0 & -1 & -1 & -1 & -1 \\
    -2 & -1 & 0 & 0 & -1 & -1 & -1 & -1 \\
    -2 & -1 & -1 & 0 & 0 & -1 & -1 & -1 
    \end{pmatrix}$}.
\]
The characteristic polynomial of $w|_{\mathbb E_7}$ is equal to the characteristic polynomial $\chi_f(t)$ of $f|_{\mathbb E_7}$ given in (\ref{eqn:a7-characteristic}), and so $w$ generates an order-$8$ cuspidal conjugacy class of $W_7$ by \cite[Lemma 3.1.10]{geck2000characters}. The group $W_7$ has a unique cuspidal conjugacy class of order $8$ by \cite[Appendix, Table B.5]{geck2000characters}, so we may assume that $f = w$. 

Consider the $\Z$-basis of $H_2(M_7, \Z)$
\[
    E_3, \, H - E_1 - E_5, \, E_2-E_4, \, E_6-E_7, \, 2H - E_1 - E_2- E_5-E_6-E_7, \, E_7, \, E_2-E_6, \, H-E_5-E_6-E_7.
\]
One can check that $f_2 = w^4$ preserves this basis, showing that there is an isomorphism of $\Z[\langle f_2 \rangle]$-modules $H_2(M_7; \Z) \cong \Z^{\oplus 2} \oplus \Z[\langle f_2 \rangle]^{\oplus 2} \oplus C^{\oplus 2}$. Here, the summands denote the trivial, regular, and cyclotomic representations of $\langle f_2 \rangle \cong \Z/2\Z$ respectively. 

By Lemma \ref{A7Lefschetz}, $\Fix(\varphi_2) \neq \emptyset$. By \Cref{edmonds}, 
\[
    b_1(\Fix(\varphi_2), \FF_2) = 2, \qquad b_0(\Fix(\varphi_2), \FF_2) + b_2(\Fix(\varphi_2), \FF_2) = 4.
\]
Noting that $\Fix(\varphi_2)$ is a finite, disjoint union of surfaces and points concludes the proof.
\end{proof}

The next two lemmas rule out the first two possibilities of Lemma \ref{A7fixedlocus1}.
\begin{lemma}\label{notXS2}
    $\Fix(\varphi_2) \not\cong X \sqcup S^2$ where $X$ is a connected surface with $\chi(X) = 0$.
\end{lemma}
\begin{proof}
Suppose that $\Fix(\varphi_2) \cong X \sqcup S^2$. By \cite[Corollary 2.6]{edmonds}, $[S^2] \neq 0 \in H_2(M_7, \Z)$ and if $X$ is orientable then $[X] \neq 0 \in H_2(M_7, \Z)$. Because $X \not\cong S^2$ and $\varphi$ preserves $\Fix(\varphi_2)$, the diffeomorphism $\varphi$ must also preserve $S^2 \subseteq \Fix(\varphi_2)$ and $X \subseteq \Fix(\varphi_2)$.

Computing with the matrix form of $f$ (as in the proof of \Cref{A7fixedlocus1}) shows that the $1$- and $(-1)$-eigenspaces of $f$ acting on $H_2(M_7, \Q)$ are spanned by $K_{M_7}$ and $\alpha := H - E_1 - E_3 - E_5$ respectively. 
\begin{enumerate}[label=(\alph*)]
\item Suppose that $X$ is orientable so that $[X]\neq 0 \in H_2(M_7, \Z)$. By the $G$-signature theorem (\Cref{gordonmain}), \Cref{A7Lefschetz}, and because $[S^2], [X]$ are disjoint and contained in $\Z\{K_{M_7}\}$ or $\Z\{\alpha\}$, 
\[
    2 = \mathrm{sign}(\varphi_2, M_7) = Q_{M_7}([S^2], [S^2]) + Q_{M_7}([X], [X]) = 2(a^2 - b^2).
\]
for some nonzero $a, b\in \Z$. There are no such integers $a, b$. 
\item Suppose that $X$ is nonorientable so that $\Fix(\varphi|_X)$ is a finite set of points (since an order-$8$ diffeomorphism cannot fix $X$). Then $\Fix(\varphi) = S^2$ or $\Fix(\varphi) = \{p_1, p_2\} \subseteq X \sqcup S^2$ by \Cref{A7Lefschetz}. 

Suppose that $\Fix(\varphi) = S^2$. By \Cref{A7Lefschetz} and the $G$-signature theorem applied to $\langle \varphi\rangle$,
\[
    2 = \csc^2\left(\frac{\pi k}{8}\right)Q_{M_7}([S^2], [S^2]) \geq Q_{M_7}([S^2], [S^2])
\]
for some odd $k \in \Z$. Since $[S^2]$ is a $1$-eigenvector for $[\varphi]$, it is contained in $\Z\{K_{M_7}\}$, and so $\csc^2\left(\frac{\pi k}{8}\right) = 1$, i.e. $k \equiv 4 \pmod 8$. This contradicts the fact that $k$ is odd, and so $\Fix(\varphi) = \{p_1, p_2\}$.

Suppose that $D\varphi$ acts on $T_{p_i}\Fix(\varphi_2)$ as an order-$m$ map for some $m$ dividing $8$ for both $i = 1, 2$. Then either $D\varphi$ or $D\varphi_4$ acts on $T_{p_i}\Fix(\varphi_2)$ by the rotation-by-$\pi$ map for both $i = 1, 2$. In other words, each $p_i$ has a tubular neighborhood that is $\langle \varphi \rangle$- or $\langle \varphi_4 \rangle$-equivariantly diffeomorphic to $(\pi, \mathbb D^2) \times (\theta_i, \mathbb D^2)$ for some $\theta_i \in \frac{\pi}{4}\Z$. Now compute that
\[
    -\sum_{i=1}^2 \cot\left(\frac{\pi}{2}\right) \cot\left(\frac{\theta_i}{2}\right) = -\sum_{i=1}^2 0 \cdot \cot\left(\frac{\theta_i}{2}\right) = 0 \neq \mathrm{sign}(\varphi^k, M_7) = 2
\]
for any $k \in \Z$ that is not divisible by $8$, by \Cref{A7Lefschetz}. This contradicts the $G$-signature theorem (\Cref{gordonmain}) applied to $\langle \varphi \rangle$ or $\langle \varphi_4 \rangle$. 

If $D\varphi$ acts on $T_{p_i}\Fix(\varphi)$ by order $m_i$ for $i = 1, 2$ with $m_1 \neq m_2$ then $p_1$ and $p_2$ are contained in different components of $\Fix(\varphi_2)$, since $\varphi$ acts on the component of $\Fix(\varphi_2)$ containing $p_i$ as an order-$m_i$ diffeomorphism. On the other hand, no finite-order diffeomorphism of $S^2$ fixes has a unique fixed point, yielding a contradiction. \qedhere
\end{enumerate}
\end{proof}

\begin{lemma}\label{notRP2RP2}
    $\Fix(\varphi_2) \not\cong \RR\PP^2 \sqcup \RR\PP^2$. 
\end{lemma}
\begin{proof}
    Suppose that $\Fix(\varphi_2) \cong \RR\PP^2 \sqcup \RR\PP^2$. Then $\varphi$ must preserve each component of $\Fix(\varphi_2)$ because $\Fix(\varphi) \neq \emptyset$. Because $\Fix(\varphi_4)$ must be orientable, $\Fix(\varphi_4) = \{p_1, p_2\}$ by \Cref{A7Lefschetz}. 

    Let $G = \langle \varphi_4 \rangle \cong \ZZ/4\ZZ$. Each $p_k$ has a tubular neighborhood $G$-equivariantly diffeomorphic to $(\theta_1(k) ,\DD^2) \times (\theta_2(k), \DD^2)$ for some $\theta_1(k), \theta_2(k) \in \frac{\pi}{2}\Z$. Because each $p_k$ is contained in a surface in $\Fix(\varphi_2)$, we may assume that $\theta_1(k) = \pi$. Then compute that
    \[
        -\sum_{k=1}^2 \cot \left(\frac{\theta_1(k)}{2} \right)\cot \left(\frac{\theta_2(k)}{2} \right) = -\sum_{k=1}^2 \cot \left(\frac{\pi}{2} \right)\cot \left(\frac{\theta_2(k)}{2} \right) = 0 \neq \mathrm{sign}(\varphi_4,M_7) = 2, 
    \]
    where the last equality follows from \Cref{A7Lefschetz}. This contradicts \Cref{gordonmain}. 
\end{proof}

The following lemma handles the last case of Lemma \ref{A7fixedlocus1}.
\begin{lemma}\label{notXp1p2}
    $\Fix(\varphi_2) \not\cong X \sqcup \{p_1,p_2\}$ for any connected surface $X$ with $\chi(X) = 0$.
\end{lemma}

Assuming Lemma \ref{notXp1p2}, we can prove the main proposition of this section.
\begin{proposition}\label{A7nonrealize}
Let $f \in W_7$ have Carter graph $A_7$. Then $f$ has order $8$ and $f$ is not realizable by any diffeomorphism $\varphi \in \Diff^+(M_7)$ of order $8$.
\end{proposition}
\begin{proof}
    The existence of such a diffeomorphism contradicts Lemmas \ref{A7fixedlocus1}, \ref{notXp1p2}, \ref{notXS2}, and \ref{notRP2RP2}.
\end{proof}

The rest of this subsection proves Lemma \ref{notXp1p2}. Suppose for the sake of contradiction that $\Fix(\varphi_2) \cong X \sqcup \{p_1, p_2\}$ for some connected surface with $\chi(X) = 0$.
\begin{lemma}\label{lem:topology-X}
The diffeomorphism $\varphi$ satisfies the following properties: 
\begin{enumerate}[label=(\alph*)]
\item $\varphi(p_k) = p_k$ for $k = 1, 2$; 
\item the group $\langle \varphi \rangle / \langle \varphi^e \rangle$ acts freely on $X$, where $e$ is the order of $\varphi|_X \in \Diff(X)$; \label{lem:X-free}
\item $\Fix(\varphi_4) = X \sqcup \{p_1, p_2\}$ with $X \cong T^2$ and $Q_{M_7}([X], [X]) = 2$. \label{lem:X-int}
\end{enumerate}
\end{lemma}
\begin{proof}
Note that $\varphi$ acts on $\Fix(\varphi_2)$ and $\Fix(\varphi) \subseteq \Fix(\varphi_2)$. Therefore $\{p_1, p_2\} \subseteq \Fix(\varphi)$ or $\Fix(\varphi) = \Fix(\varphi|_X)$. 

Suppose that $\Fix(\varphi) = \Fix(\varphi|_X)$. By Lemma \ref{A7Lefschetz}, $\chi(\Fix(\varphi|_X)) = 2$ and so $\varphi$ fixes two points in $X$. Also because $p_1, p_2$ are contained in $\Fix(\varphi_4)$ and $\Fix(\varphi_4|_X) \neq \emptyset$, Lemma \ref{A7Lefschetz} shows that $X \subseteq \Fix(\varphi_4)$. Because $\varphi_4$ has order greater than $2$, $X$ is orientable; because $\varphi$ fixes isolated points in $X$, $\varphi|_X$ is orientation-preserving of order $2$. The Riemann--Hurwitz formula implies that such a diffeomorphism of $X \cong T^2$ does not exist, which is a contradiction. Therefore, $\{p_1, p_2\} \subseteq \Fix(\varphi)$. This also implies that for any $k \in \Z/8\Z$,  $\Fix(\varphi^k|_X) = X$ or $\Fix(\varphi^k|_X) = \emptyset$ by \Cref{A7Lefschetz}. In other words, $\langle \varphi \rangle/ \langle \varphi^e \rangle$ acts freely on $X$. 

Suppose that $\Fix(\varphi_4) = \{p_1, p_2\}$. By \Cref{A7Lefschetz} and the $G$-signature theorem (\Cref{gordonmain}) applied to $\langle \varphi\rangle$ and $\langle \varphi_4 \rangle$ respectively, there exist $\theta_i(j) = \frac{2\pi a_{i,j}}{8}$ for $i,j = 1, 2$ and odd integers $a_{i,j} \in \Z$ so that
\begin{align*}
- 2 = \cot\left(\frac{\theta_1(1)}{2}\right)\cot\left(\frac{\theta_2(1)}{2}\right) + \cot\left(\frac{\theta_1(2)}{2}\right) \cot\left(\frac{\theta_2(2)}{2}\right)= \cot ( \theta_1(1))\cot(\theta_2(1)) + \cot(\theta_1(2)) \cot(\theta_2(2)).
\end{align*}
Because $\cot(\theta_i(j)) = \pm 1$ for all $i, j = 1, 2$, the second equation shows without loss of generality that $\theta_2(j) = \theta_1(j) + \frac \pi 2$ for both $j = 1, 2$. For any choice of odd $a_{1,j} \in \Z$,
\[
    \cot\left(\frac{\theta_1(j)}{2}\right)\cot\left(\frac{\theta_2(j)}{2}\right) = \cot\left(\frac{\pi a_{1,j}}{8}\right) \cot\left(\frac{\pi (a_{1,j} + 2)}{8}\right) = 1 \quad \text{ or }\quad -3 \pm 2 \sqrt 2.
\]
Therefore, $\{p_1, p_2\} \subsetneq \Fix(\varphi_4)$; by \ref{lem:X-free}, $\Fix(\varphi_4) = X \sqcup \{p_1, p_2\}$. Because $\varphi_4$ has order $4$, the surface $X$ is orientable. Finally, apply $G$-signature theorem (\Cref{gordonmain}) to the $\langle \varphi_2 \rangle$-action on $M_7$ to see that $Q_{M_7}([X], [X]) = 2$. 
\end{proof}

Let $B_1, B_2 \cong B^4 \subseteq M_7$ be $\langle \varphi \rangle$-equivariant open neighborhoods of $p_1,p_2 \in M_7$ so that $X$ is contained in $M_7 - (B_1 \sqcup B_2)$ and $\langle \varphi \rangle$ acts freely on $M_7 - (B_1 \sqcup B_2\sqcup X)$. Then \cite[Section 4.3.1]{morita} shows that
\[
	M := \left(M_7 - (B_1 \sqcup B_2)\right)/\langle \varphi \rangle
\]
is a $4$-manifold with boundary $\partial M = L(8, a_1) \sqcup L(8, a_2)$ for some $a_1, a_2 \in \Z$, a disjoint union of lens spaces. By \cite[Exercise 5.3.9(b), Example 4.6.2]{gompf--stipsicz}, there exists a $2$-handlebody $P(8,a_k)$ with $\partial P(m,a_k) \cong L(8, a_k)$ for each $k = 1, 2$. By \cite[Corollary 5.3.12]{gompf--stipsicz}, 
\[
    \pm\det(Q_{P(8,a_k)}) = \lvert H_1(\partial P(8,a_k), \Z) \rvert = 8. \qedhere
\]  
Let $P := P(8, a_1) \sqcup P(8, a_2)$ and let $A$ be the closed $4$-manifold
\[
    A := M \cup_{\partial M} P.
\]

Let $q: M_7 \to M_7^*:= M_7 / \langle \varphi \rangle$ denote the quotient map. Consider the composition $q: q^{-1}(M) \to M \hookrightarrow A$. In the following three lemmas we analyze the topology of $A$ in comparison to the topology of $M$.
\begin{lemma}\label{lem:euler-A}
$\chi(A)= 3 + b_2(P)$. 
\end{lemma}
\begin{proof}
Note that $\chi(\Fix(\varphi_2)) = 2$ and that $\langle \varphi \rangle$ acts freely on $M_7 - \Fix(\varphi_2)$. Moreover, $\chi(q(X)) = 0$, and so by multiplicativity of Euler characteristic and by inclusion-exclusion,
\[
    \chi(M) = \frac{\chi(M_7) - \chi(\Fix(\varphi_2))}{\lvert \langle \varphi \rangle \rvert} + \chi(q(X)) = \frac{10 - 2}{8} + 0 = 1. 
\]
Because $P$ is a disjoint union of two $2$-handlebodies, $\chi(P) = 2 + b_2(P)$. By inclusion-exclusion, 
\[
    \chi(A) = \chi(M) + \chi(P) = 3 + b_2(P). \qedhere
\]
\end{proof}

\begin{lemma}\label{lem:armstrong-apps}
There are isomorphisms 
\[
    \pi_1(M) \cong \Z/e\Z, \qquad \pi_1(M_7^*) \cong \pi_1(B_1/\langle \varphi \rangle) \cong \pi_1(B_2/\langle \varphi \rangle) \cong 1.
\]
\end{lemma}
\begin{proof}
Suppose that $\psi$ is a finite-order homeomorphism of a simply-connected manifold $Y$ with $\Fix(\psi) \neq \emptyset$. According to Armstrong \cite[Example 4]{armstrong}, the quotient $Y / \langle \psi \rangle$ is simply-connected, which applied to the action of $\varphi$ on $M_7$ and $B_k$ shows that $\pi_1(M_7^*) = 1$ and $\pi_1(B_k/\langle \varphi \rangle) = 1$ for $k = 1, 2$. On the other hand, $\Fix(\varphi^e|_{q^{-1}(M)}) = X$, and so $q^{-1}(M)/\langle \varphi^e \rangle$ is simply-connected. The group $\langle \varphi \rangle / \langle \varphi^e \rangle$ acts freely on $q^{-1}(M)/\langle \varphi^e \rangle$ and the quotient is $M$. Therefore, $\pi_1(M) \cong \Z/e\Z$.
\end{proof}

\begin{lemma}\label{lem:self-int-X}
The image $q(X)$ is a submanifold of $A$ and
\[
    Q_A([q(X)], [q(X)]) = 16 e^{-2}. 
\]
\end{lemma}
\begin{proof}
First consider the intermediate quotient 
\[
    q': (M_7 - (B_1 \sqcup B_2)) \to M' := (M_7 - (B_1 \sqcup B_2))/\langle \varphi^e \rangle
\] 
The group $\langle \varphi^e \rangle$ of order $\frac 8e$ acting on $M_7 - (B_1 \sqcup B_2)$ fixes $X$ pointwise and acts freely on the complement of $X$. By \cite[Section 4.3.1]{morita}, $M'$ is a smooth $4$-manifold with boundary, $q'(X)$ is a submanifold of $M'$, and if $\nu_{M_7}(X)$ is the normal bundle of $X$ in $M_7$ then $\nu_{M_7}(X)^{\otimes 8/e}$ is the normal bundle $\nu_{M'}(q'(X))$ of $q'(X)$ in $M'$, viewing both oriented $\R^2$-bundles over $X$ as $\C$-line bundles over $X$.

The group $H := \langle \varphi \rangle / \langle \varphi^e \rangle $ acts on $M'$ and acts freely on $q'(X)$ so it induces a free action on the normal bundle $\nu_{M'}(q'(X))$. Let $q'': M' \to M$ denote the quotient by $H$ so that $q'' \circ q' = q$. There is an isomorphism of $\C$-line bundles
\[
    (q'')^*\nu_M(q(X)) \cong \nu_{M'}(q'(X))
\]
induced by the derivative $Dq''$. Taking Chern classes and noting that $q''$ is a map of degree $e$, compute
\[
    e\, c_1(\nu_{M}(q(X))) = c_1((q'')^*\nu_{M}(q(X))) = c_1(\nu_{M'}(q'(X))) = c_1(\nu_{M_7}(X)^{\otimes 8/e}) = \frac 8e c_1(\nu_{M_7}(X)) \in \Z,
\]
where we have identified $H^2(X) \cong H^2(q(X)) \cong \Z$. Because the first Chern class is the Euler class for $\C$-line bundles and because $Q_{M_7}([X], [X])= 2$ by Lemma \ref{lem:topology-X}\ref{lem:X-int}, 
\[
    \frac{16}{e^2} = \frac{8}{e^2} Q_{M_7}([X], [X]) = \frac{8}{e^2} c_1(\nu_{M_7}(X)) = c_1(\nu_M(q(X))) = Q_{M}([q(X)], [q(X)]). \qedhere
\]
\end{proof}

Lemma \ref{lem:self-int-X} will be used in an application of the following algebraic lemma. 
\begin{lemma}\label{lem:unimodular-index}
Let $(L, Q)$ be a unimodular (nondegenerate, bilinear, symmetric, integral) lattice. Let $L_0 \leq L$ be a subgroup of finite index and consider the restriction $Q|_{L_0}$. Then $\det(Q|_{L_0})= \pm [L:L_0]^2$
\end{lemma}
\begin{proof}
Identify $L \cong \Z^m$ and consider the matrix $M \in \mathrm{Mat}_{m \times m}(\Z)$ sending a $\Z$-basis $e_1, \dots, e_m$ of $L$ to a $\Z$-basis $f_1, \dots, f_m \in L$ of $L_0$. If $A$ is the matrix form of the form $Q$ with respect to the $\Z$-basis $e_1, \dots, e_m$ then $M^T A M$ is the matrix form of the restriction $Q|_{L_0}$ with respect to the $\Z$-basis $f_1, \dots, f_m$. Taking determinants, compute that $\det(Q|_{L_0}) = \pm \det(M)^2$ because $\det(A) = \pm 1$ by unimodularity of $(L, Q)$. Finally, taking the Smith normal form of the matrix $M$ shows that $\lvert \det(M)\rvert = [L : L_0]$. 
\end{proof}

With the computation of Lemmas \ref{lem:euler-A} and \ref{lem:self-int-X} in hand, we analyze the topology of $A$ conclude the proof of Lemma \ref{notXp1p2}.
\begin{proof}[{Proof of Lemma \ref{notXp1p2}}]
Consider the Mayer--Vietoris sequence for the union $M_7^* = M \cup \left((B_1 \sqcup B_2) / \langle \varphi \rangle\right)$,
\[
    H_1(\partial M) \to H_1(M) \oplus H_1((B_1 \sqcup B_2)/\langle \varphi \rangle)\to H_1(M_7^*).
\]
Because $H_1(M_7^*) = H_1((B_1 \sqcup B_2) / \langle \varphi \rangle) = 0$ by Lemma \ref{lem:armstrong-apps}, $i_*: H_1(\partial M) \to H_1(M)$ is surjective. 

The following sequence is exact by the Mayer--Vietoris sequence for $A = M \cup P$:
\[
    H_2(\partial M) \to H_2(M) \oplus H_2(P) \xrightarrow{F} H_2(A) \to H_1(\partial M) \to H_1(M) \oplus H_1(P) \to H_1(A) \to 0. 
\]
Since $H_1(P) = 0$ and $H_1(\partial M) \to H_1(M)$ is surjective, $H_1(A) = 0$. An application of the universal coefficient theorem shows that $H^2(A)$ is torsion-free, and hence $H_2(A) \cong \Z^{1 +b_2(P)}$ by Lemma \ref{lem:euler-A}. 

Because $H_2(\partial M) = 0$ and $H_1(\partial M)$ is finite, the map $F$ is injective and has finite-index image in $H_2(A)$. Let $F'$ denote the injection
\[
    F' : \Z\{[q(X)]\} \oplus H_2(P) \hookrightarrow H_2(M)\oplus H_2(P) \xrightarrow{F} H_2(A),
\]
which also has finite-index image by rank reasons. By Lemmas \ref{lem:unimodular-index} and \ref{lem:self-int-X},
\[
    [H_2(A): \mathrm{im}(F')]^2 = \lvert \det(Q|_{\Z\{[X]\} \oplus H_2(P)})\rvert = \lvert Q_A([q(X)], [q(X)])\cdot\det(Q_{P(8, a_1)}\oplus Q_{P(8, a_2)}) \rvert  = (32 e^{-1})^2.
\]
Because $\mathrm{im}(F') \leq \mathrm{im}(F)$, the index $[H_2(A) : \mathrm{im}(F)]$ must divide $[H_2(A): \mathrm{im}(F')] = 32 e^{-1}$.

By Lemma \ref{lem:armstrong-apps}, $H_1(M) = \Z/e\Z$. By exactness of the Mayer--Vietoris sequence, 
\[
    [H_2(A) : \mathrm{im}(F)] = \lvert \ker(i_*: H_1(\partial M) \to H_1(M)) \rvert = 64 e^{-1}, 
\]
which does not divide $32e^{-1}$, yielding a contradiction.
\end{proof}

\subsection{Proof of \Cref{mainthm1}}\label{sec:mainthmproof}
With the results of this section in hand, we are ready to prove \Cref{mainthm1}.
\begin{proof}[Proof of \Cref{mainthm1}]
The equivalence of (a), (b), and (c) is established in \Cref{MCequiv}. Complex automorphisms are smooth, so (b) implies (d). It remains to prove that (d) implies (b), which we do below by contrapositive. 

First, suppose that $f$ is contained in $W_n \subset \Mod^+(M_n)$ and that $f$ is not realizable by complex automorphisms of any del Pezzo surface $(M_n, J)$. Then \Cref{CarterClassify} shows that either $n = 5$ and $f$ has Carter graph $D_5(a_1)$ or $D_2 + D_3$, or $n = 7$ and $f$ has Carter graph $A_7$, $D_4 +3A_1$, $D_6 +A_1$, or $E_7(a_3)$. In each case, $f$ is not smoothly realizable by Propositions \ref{D2D3}, \ref{D5a1}, \ref{A7nonrealize}, \ref{D6A1nonrealize}, \ref{D43A1nonrealize}, and \ref{E7a3nonrealize} respectively. Therefore, the direction (d) implies (b) holds if $f$ is contained in $W_n$. 

Suppose that $f$ is not contained in $W_n \subseteq \Mod^+(M_n)$ and that $f$ is smoothly realizable by a diffeomorphism $\varphi \in \Diff^+(M_n)$ of order $m$. Apply \Cref{irredclassification} to see that there exists $h \in \Mod^+(M_n)$ so that $hfh^{-1}$ is contained in $W_n$. There exists $\psi \in \Diff^+(M_n)$ with $[\psi] = h$ by \cite[Theorem 2]{wall-diffeos}, and hence $h f h^{-1}$ is smoothly realizable by $\psi\circ \varphi \circ \psi^{-1}$. By the previous paragraph, there exists a complex structure $(M_n, J)$ of a del Pezzo surface and an automorphism $\Phi \in \Aut(M_n, J)$ of order $m$ realizing $hfh^{-1}$. Finally, $\psi^{-1} \circ \Phi \circ \psi$ is an automorphism of $\Aut(M_n, \psi^* J)$ realizing $f$.
\end{proof} 

\section{Coxeter elements and complex Nielsen realization}\label{sectionCoxeter}

In this section we study complex realizability and irreducibility of the Coxeter elements and, more generally, other elements of order equal to the Coxeter number of $W_n$. 

\subsection{Coxeter elements, irreducibility, and realizability}\label{sec:coxeter-irred-real}
Because $W_n$ is a Coxeter group, it admits a distinguished conjugacy class of \emph{Coxeter elements}. In this subsection we characterize the (conjugates of the) Coxeter elements of $\Mod^+(M_n)$ for $3 \leq n \leq 8$ via irreducibility and realizability. 
\begin{definition}\label{coxeterDef}
    A \emph{Coxeter element} $w \in W_n$ is any product of the simple reflections, taken one at a time in any order. All Coxeter elements are conjugate in $W_n$ \cite[Proposition 3.16]{humphreys}, and the \emph{Coxeter number} of $W_n$ is defined to be the order of any Coxeter element. We denote the Coxeter number of $W_n$ by $h_n$.
\end{definition}
According to \cite[Sections 3.16, 8.4]{humphreys} the Coxeter numbers of $W_n$ are $h_n = 6,5,8,12,18,30$ for $n = 3,4,5,6,7,8$ respectively, and $h_n = \infty$ for $n>8$. According to \cite[(2.4)]{mcmullen}, the characteristic polynomial $\chi_w(t)$ of a Coxeter element $w \in W_n$ acting on the geometric representation $\mathbb E_n$ of $W_n$ is 
\begin{equation}\label{eqn:char-poly}
\chi_w(t) = \frac{t^{n-2}(t^3-t-1) + (t^3 +t^2-1)}{t-1}. 
\end{equation}

The following theorem distinguishes the Coxeter elements of $W_n$ among all elements of order $h_n$ via irreducibility for $3 \leq n \leq 7$. 
\begin{theorem}\label{cox_conjugate_7}
    Let $3 \leq n \leq 7$. An element $w \in \Mod^+(M_n)$ is an irreducible class of order $h_n$ if and only if $w$ is conjugate in $\Mod^+(M_n)$ to a Coxeter element of $W_n$.
\end{theorem}
\begin{proof}
    Suppose that $w$ is conjugate in $\Mod^+(M_n)$ to a Coxeter element of $W_n$. The $1$-eigenspace of $w$ on $H_2(M_n, \Z)$ is $\ZZ \{K_n\}$ by \cite[Lemma 3.16]{humphreys}, so $w$ is irreducible by \Cref{posdefcriterionirred}.

    Suppose that $w \in \Mod^+(M_n)$ is irreducible and has order $h_n$. By \Cref{irredclassification}, $w$ is conjugate in $\Mod^+(M_n)$ into the stabilizer $\Stab(K_{M_n}) = W_n$ of $K_{M_n}$.

    For the $n=3$ case: $W_3 \cong S_3 \times \Z/2\Z$. A Coxeter element is the product of the permutation $(1,2,3)$ and a generator of the $\Z/2$ factor. There is only one conjugacy class of order $3$ in $S_3$. Hence, there is only one conjugacy class of order $6$ elements in $W_3$.

    For the $n=4$ case: $W_4 \cong S_5$. There is only one conjugacy class of order $5$ elements in $W_4$.

    For the $n=5$ case: by \Cref{irredclassification}, we only need to consider cuspidal classes in $W_5 \cong W(D_5)$ and $P_5 \cong W(D_4)$, which are in bijection with even partitions of $5$ and $4$ respectively, by \Cref{negativeEven}. The even partitions of $4$ with signed cycle-type are $[\bar{1},\bar{1},\bar{1},\bar{1}], [\bar{2},\bar{2}], [\bar{1},\bar{3}]$. The even partitions of $5$ are in correspondence with the signed cycle-types $[\bar{2},\bar{1},\bar{1},\bar{1}], [\bar{1},\bar{4}], [\bar{2},\bar{3}]$. \Cref{negativeEven} implies that the only class with order $h_5 = 8$ is $[\bar{1}, \bar{4}]$, the conjugacy class of the Coxeter elements.

    For the $n=6$ case: we only need to check the cuspidal conjugacy classes of $W_6 = W(E_6)$ by Lemma \ref{4and6iff}. These are listed in \cite[Appendix, Table B.4]{geck2000characters}, and there is only one cuspidal class of order $h_6 = 12$, given by the Coxeter elements.

    For the $n=7$ case: there is only one cuspidal class of order $h_7=18$ elements in $W_7 = W(E_7)$ by \cite[Appendix, Table B.5]{geck2000characters}, given by the Coxeter elements. Any irreducible element in $W_7$ with $2$-dimensional $1$-eigenspace conjugates to a cuspidal representative in $P_7 = W(D_6) \subset W_7$ by \Cref{irredclassification}. These classes are in bijection with even partitions of $6$, whose associated signed cycles are 
    $$[\bar{1},\bar{5}], [\bar{2},\bar{4}], [\bar{3},\bar{3}], [\bar{2},\bar{2},\bar{1},\bar{1}], [\bar{3},\bar{1},\bar{1},\bar{1}], [\bar{1},\bar{1},\bar{1},\bar{1},\bar{1},\bar{1}],$$
    with orders $10$, $8$, $6$, $4$, $6$, $2$, respectively by \Cref{negativeEven}, none of which are $18$.
\end{proof}

The next result classifies the irreducible classes of $\Mod^+(M_8)$ of order $h_8$. Before we state the result, first consider the element $r \in \Mod^+(M_8)$ of order $h_8 = 30$ and its power $r_3 := r^{10}$ of order $3$, defined with respect to the usual basis $(H, E_1, \dots, E_8)$ of $H_2(M_7, \Z)$.
\begin{equation}\label{eqn:r-nonzerotrace}
	r = \scalebox{0.75}{$\begin{pmatrix}
	3 & 1 & 1 & 0 & 0 & 0 & 1 & 1 & 2 \\
	-1&-1 & 0 & 0 & 0 & 0 & 0 & 0 & -1 \\
	-1 & 0 & -1 & 0 & 0 & 0 & 0 & 0 & -1 \\
	-2 & -1 & -1 & 0 &  0 & 0 & -1 & -1 & -1\\
	-1 & 0 & 0 & 0 & 0 & 0 & 0 & -1 & -1 \\
	-1 & 0 & 0 & 0 & 0 & 0 & -1 & 0 & -1 \\
	0 & 0 & 0 & 1 & 0 & 0 & 0 & 0 & 0 \\
	0 & 0 & 0 & 0 & 1 & 0 & 0 & 0 & 0 \\
	0 & 0 & 0 & 0 & 0 & 1 & 0 & 0 & 0
	\end{pmatrix}$}, \qquad r_3 = r^{10} = \scalebox{0.75}{$\begin{pmatrix}
	2 & 0 & 0 & 1  & 0 & 1 & 0 & 0 & 1 \\	
	0 & 1 & 0 & 0  & 0 & 0 & 0 & 0 & 0 \\
	0 & 0 & 1 & 0  & 0 & 0 & 0 & 0 & 0 \\
	-1 & 0& 0 & 0  & 0 & -1& 0 & 0 & -1\\
	0 & 0 & 0 & 0  & 1 & 0 & 0 & 0 & 0 \\
	0 & 0 & 0 & 0  & 0 & 0 & 1 & 0 & 0\\
	-1 & 0& 0 & -1 & 0 & 0 & 0 & 0 & -1\\
	0 & 0 & 0 & 0  & 0 & 0 & 0 & 1 & 0\\
	-1 & 0& 0 & -1 & 0 & -1& 0 & 0 & 0
	\end{pmatrix}$}.
\end{equation}
\begin{theorem}\label{cox_conjugate_8}
    An element $w \in \Mod^+(M_8)$ is irreducible of order $h_8 = 30$ if and only if $w$ is conjugate in $\Mod(M_8)$ to a Coxeter element of $W_8$ or to $r \in W_8$. 
\end{theorem}
\begin{proof}
    Let $c \in W_8$ denote a Coxeter element of $W_8$. By (\ref{eqn:char-poly}), the characteristic polynomial $\chi_c(t)$ of $c|_{\mathbb E_8}$ is the cyclotomic polynomial
    \[
        \chi_c(t) = t^8 +t^7-t^5-t^4-t^3 +t+1,
    \]
    so $H_2(M_8,\Z)^{\langle c \rangle} = \Z\{K_{M_8}\}$. Note that $r$ fixes $K_{M_n}$ and the characteristic polynomial $\chi_r(t)$ of $r|_{\mathbb E_8}$ is
    \begin{equation}\label{eqn:char-r}
        \chi_r(t) = (t+1)^2 (t^2-t+1) (t^4-t^3+t^2-t+1),
    \end{equation}
    and so $H_2(M_8, \Z)^{\langle r \rangle} = \Z\{K_{M_8}\}$ and $r$ has order $30$. 
    
    Suppose that $w \in \Mod^+(M_8)$ is conjugate in $\Mod(M_8)$ to $c$ or to $r$. Then the $1$-eigenspaces of $c$ and $r$ are $\Z\{K_{M_8}\}$, so \Cref{posdefcriterionirred} implies that $w$ is irreducible.

    Suppose $w \in \Mod^+(M_8)$ is irreducible of order $30$. By \Cref{irredclassification}, $w$ is conjugate to a cuspidal class of $W_8$ or a cuspidal class of $P_8$. By \Cref{negativeEven}, the cuspidal classes of $P_8=W(D_7) \subset W_8$ have one of the following signed cycle-types
    \[
        [\bar{1},\bar{6}], \quad [\bar{2},\bar{5}],\quad  [\bar{3},\bar{4}],\quad  [\bar{2},\bar{2},\bar{2},\bar{1}],\quad  [\bar{3},\bar{2},\bar{1},\bar{1}],\quad [\bar 4, \bar 1, \bar 1, \bar 1], \quad  [\bar{2},\bar{1},\bar{1},\bar{1},\bar{1},\bar{1}],
    \]
    and have orders $12,20,24,4,12,8$ or $4$. Because $w$ has order $30$, it is not conjugate in $\Mod(M_8)$ to a cuspidal class of $P_8$. 
    
    There are two cuspidal conjugacy classes of order $30$ in $W_8=W(E_8)$ by \cite[Appendix, Table B.6]{geck2000characters}. By eigenvalue considerations, $c$ and $r$ are not conjugate in $\Mod(M_8)$. Moreover, the classes $c$ and $r$ generate cuspidal conjugacy classes in $W_8$ by \cite[Lemma 3.1.10]{geck2000characters}. Therefore, $w$ is conjugate in $\Mod(M_8)$ to one of these two cuspidal classes, $c$ or $r$.
\end{proof}

The following proposition distinguishes the Coxeter elements of $W_8$ from the class $r \in W_8$ by their smooth realizability.
\begin{proposition} \label{30_irred_tr1_nonrealizable}
    There is no finite order diffeomorphism $\varphi \in \Diff^+(M_8)$ with $[\varphi] = r$.
\end{proposition}
\begin{proof}
Consider the isomorphism of groups 
\[
	H_2(M_8, \Z) \cong \Z\{H-E_3, \, H - E_8, \, E_1, E_2,\, E_4,\, E_7\} \oplus \Z\{E_3, \, H - E_6 - E_8, \, H- E_5 - E_8\}.
\]
Using the matrix form of $r_3$, we can compute that $\langle r_3\rangle$ acts trivially on the first summand and that $\langle r_3\rangle$ acts by the regular representation on the second summand.

Let $r_6 := r^5$ so that $r_3 = r_6^2$. Note that $H_2(M_8, \Z)$ has no cyclotomic summands as a $\Z[\langle r_3 \rangle]$-module. Furthermore, by eigenvalue considerations (using the characteristic polynomial $\chi_r(t)$ of $r|_{\mathbb E_8}$ given in (\ref{eqn:char-r})), the characteristic polynomial $\chi_{r_6}(t)$ of $r_6|_{\mathbb E_8}$ is
\[
    \chi_{r_6}(t) = (t+1)^{6} (t^2 - t +1).
\]
The Lefschetz number of $r_6$ is $\Lambda(r_6) = -2 < 0$ and by \Cref{ObstructionLemma}, there is no diffeomorphism of finite order that represents $r_6$. Therefore, there does not exist any diffeomorphism $\varphi$ of finite order with $[\varphi] = r$. 
\end{proof}

With the above understanding of irreducible, order-$30$ elements of $\Mod(M_8)$, we are ready to prove \Cref{coxetermainthm}.
\begin{proof}[Proof of \Cref{coxetermainthm}]
For $n = 3$, note that the group $W_3 \leq \Mod^+(M_3)$ is the image of $\Aut(M_3, J)$ under the map $\Diff^+(M_3) \to \Mod(M_3)$ admitting a section $W_3 \to \Aut(M_3, J)$ by \cite[Theorem 8.4.2]{dolgachev}, where $J$ is the unique complex structure on $M_3$ so that $(M_3, J)$ is del Pezzo surface, up to isomorphism. In particular, any Coxeter element $w \in W_3$ is realizable by an automorphism of a del Pezzo surface $(M_3, J)$. 

For $4 \leq n \leq 8$, consider the \emph{standard} Coxeter element $w \in W_n$ (cf. \cite[Section 8]{mcmullen})
\[
    w = \RRef_{E_1-E_2}\circ \RRef_{E_2-E_3} \circ \dots \circ \RRef_{E_{n-1} - E_n}\circ \RRef_{H-E_1-E_2-E_3}.
\]
McMullen \cite[Theorem 11.1]{mcmullen} shows that for certain choices of $(a, b) \in \C^2$, the birational map $f: \CP^2 \dashrightarrow \CP^2$ given in affine coordinates by 
\begin{equation}
    (x,y) \mapsto (a,b) + (y, y/x)
\end{equation}
induces a complex automorphism $\varphi \in \Aut(M_n, J)$ for some complex structure $(M_n, J) \cong \Bl_{p_1, \dots, p_n}\CP^2$ such that $[\varphi] = w$.

Suppose that $f \in \Mod^+(M_n)$ is an irreducible class of order $h_n$, and if $n = 8$ assume further that $g$ has trace $0$. By Theorems \ref{cox_conjugate_7} and \ref{cox_conjugate_8}, there exists $h \in \Mod^+(M_n)$ so that $f = h^{-1}wh$. There exists $\psi \in \Diff^+(M_n)$ with $h = [\psi]$ by \cite[Theorem 2]{wall-diffeos} and $\psi^{-1} \circ \varphi \circ \psi$ is an order-$h_n$ complex automorphism of $(M_n, \psi^*J)$ such that $[\psi^{-1} \circ \varphi \circ \psi] = f$. Because $f$ is irreducible, Lemma \ref{irred-minimal} and \cite[Theorem 3.8]{dolgachev--iskovskikh} shows that $(M_n, \psi^*J)$ is a del Pezzo surface.

If $n = 8$ and $f$ has nonzero trace then Theorem \ref{cox_conjugate_8} shows that there exists $h \in \Mod^+(M_n)$ so that $f = h^{-1} r h$ where $r \in W_8$ is as defined in (\ref{eqn:r-nonzerotrace}). There exists $\psi \in \Diff^+(M_n)$ with $h = [\psi]$ by \cite[Theorem 2]{wall-diffeos}, so $f$ is realizable by a finite-order diffeomorphism if and only if $r$ is. However, Proposition \ref{30_irred_tr1_nonrealizable} shows that $r$ is not realizable by any finite-order diffeomorphism.
\end{proof}

\subsection{Irreducibles of prime order}

The following result gives a more refined characterization of the prime order irreducible mapping classes on del Pezzo manifolds than \Cref{CarterClassify}. 

\begin{theorem}\label{primeIrredCuspidal}
    If $n = 3, 5,$ or $7$ then there does not exist an irreducible element of odd, prime order in $W_n$. There is one conjugacy class of irreducible elements of odd, prime order in $W_4$, represented by the Coxeter elements. There is one conjugacy class of odd, prime order irreducible elements in $W_6$, represented by the fourth power of a Coxeter element. There are two conjugacy classes of odd, prime order irreducible elements in $W_8$, both represented by powers of a Coxeter element.
\end{theorem}
\begin{proof}
    For $3 \leq n \leq 8$, by \Cref{irredclassification} and \Cref{negativeEven}, it suffices to consider cuspidal conjugacy classes in $W_n$, since all cuspidal conjugacy classes of $P_n$ have even order. 
    
    For $n=3$, we have $W_3 \cong W(A_2) \times W(A_1)$. There is only one cuspidal conjugacy class in $W(A_n)$, given by the Coxeter elements, and by \cite[Exercise 3.10]{geck2000characters} there is only one cuspidal conjugacy class in $W_3$, represented by the Coxeter elements of $W_3$. Such elements have order $6$. 
    
    For $n=4$, there is only one cuspidal conjugacy class in $W_4 \cong W(A_4)$, represented by Coxeter elements, which have order $5$.
    
    There are no odd-order cuspidal classes in $W_5 \cong W(D_5)$ and $W_7 \cong W(E_7)$, by \Cref{negativeEven} and \cite[Appendix, Table B.5]{geck2000characters} respectively. 

    According to \cite[Appendix, Tables B.4, B.6]{geck2000characters}, there are two odd-order cuspidal classes (of orders $9$ and $3$) in $W_6$ and there are four odd-order cuspidal classes (of orders $15$, $5$, $9$, and $3$) in $W_8$. 

    Let $n = 6$ or $8$. The characteristic polynomial $\chi_n(t)$ of a Coxeter element $w \in W_n$ acting on $\mathbb E_n$ is 
    \[
        \chi_6(t) = \Phi_3(t)\Phi_{12}(t), \qquad \chi_8(t) = \Phi_{30}(t)
    \]
    by (\ref{eqn:char-poly}), where $\Phi_m(t)$ denotes the $m$th cyclotomic polynomial. Note that $\zeta^{h_n/p} \neq 1$ for any odd prime $p$ dividing $h_n$ and for any root $\zeta$ of $\chi_n(t)$. So for any such prime $p$, the $1$-eigenspace of $w^{h_n/p}$ acting on $H_2(M_n, \Z)$ is $\Z \{ K_{M_n} \}$, and $w^{h_n/p}$ generates the unique order-$p$ cuspidal class of $W_n$ by \cite[Lemma 3.1.10]{geck2000characters}. 
\end{proof}

Below, we observe that Coxeter elements also characterize irreducible involutions $f \in \Mod^+(M_n)$ with the additional assumption that $H_2(M_n, \Z)^{\langle f \rangle} \cong \Z$. In this case, we recover a partial version of \cite[Theorem 1.3]{lee-involutions}. The proof below differs from the enumerative proof of \cite{lee-involutions} in light of Lemma \ref{irredclassification}. The proof idea is similar to that of \cite[Theorem 1.4]{bayle--beauville} but replaces a key Mori theory input with Lemma \ref{irredclassification}.

\begin{theorem}[{cf. \cite[Theorem 1.3]{lee-involutions}}]\label{geiserbertini}
Let $3 \leq n \leq 8$ and consider $f \in \Mod^+(M_n)$ of order $2$. Then $f$ is irreducible and satisfies $H_2(M_n, \Z)^{\langle f \rangle} \cong \Z$ if and only if $f$ is conjugate in $\Mod(M_n)$ to the power $w^{\frac{h_n}{2}}$ of a Coxeter element $w \in W_n \leq \Mod(M_n)$ and $n = 7$ or $8$. 
\end{theorem}
\begin{proof}
Let $w \in W_n$ denote a Coxeter element. First, we show that $w^{\frac{h_n}{2}}$ is irreducible and that $H_2(M_n, \Z)^{\langle w^{\frac {h_n}{2}} \rangle} \cong \Z$ if $n = 7$ or $8$. To this end, consider $-I_n \in \Mod(M_n)$, the class acting by negation on $H_2(M_n, \Z)$. If $n = 7$ or $8$ then $Q_{M_n}(K_{M_n}, K_{M_n}) = 2$ or $1$ respectively, and so the reflection $\RRef_{K_n}$ is well-defined. Then $-I_n \circ \RRef_{K_{M_n}}$ is contained in $W_n$ and acts by negation on $\mathbb E_n$. Because an element of $W_n$ acts by negation on $\mathbb E_n$, \cite[Corollary 3.19]{humphreys} says that the power $w^{\frac{h_n}{2}}$ acts by negation on $\mathbb E_n$ and $w^{\frac{h_n}{2}} = -I_n \circ \RRef_{K_{M_n}}$. Finally, note that $H_2(M_n, \Z)^{\langle w^{\frac {h_n}{2}} \rangle} = \Z\{K_{M_n}\}$. By \Cref{posdefcriterionirred}, $w^{\frac{h_n}{2}}$ is irreducible. Any $\Mod(M_n)$-conjugate of $w^{\frac{h_n}{2}}$ is also irreducible.

To prove the converse, suppose that $f \in \Mod^+(M_n)$ is irreducible and satisfies $H_2(M_n, \Z)^{\langle f \rangle} \cong \Z$. By Lemma \ref{irredclassification}, $f$ is conjugate in $\Mod^+(M_n)$ to an element of $W_n$. After possibly replacing $f$ with a $\Mod^+(M_n)$-conjugate of $f$, this implies that $f$ acts on $\mathbb E_n$ with no non-zero fixed points, and so the characteristic polynomial $\chi_f(t)$ of $f|_{\mathbb E_n}$ is $\chi_f(t) = (t+1)^n$. Since $f|_{\mathbb E_n}$ has finite order, it is diagonalizable over $\C$, i.e. by eigenvalue considerations, $f$ is conjugate in $\GL(\mathbb E_n \otimes \C)$ to $-I_n|_{\mathbb E_n}$, the negation map restricted to $\mathbb E_n$. In other words, $f|_{\mathbb E_n} = -I_n|_{\mathbb E_n}$, and $f = w^{\frac{h_n}{2}}$ by \cite[Corollary 3.19]{humphreys} again. Moreover, $-I_n \circ f$ fixes $\mathbb E_n$ pointwise while negating $K_{M_n}$, i.e. $-I_n \circ f = \RRef_{\alpha}$ for some $\alpha \in \Z\{K_{M_n}\}$ and $Q_{M_n}(\alpha, \alpha) = \pm 1$ or $\pm 2$. Since $Q_{M_n}(K_{M_n}, K_{M_n}) = 9-n$, we conclude that $n = 7$ or $8$ and $\alpha = K_{M_n}$.
\end{proof}

\begin{remark}
\cite[Theorem 1.3]{lee-involutions} further shows that for $1 \leq n \leq 8$, any irreducible involution $f \in \Mod^+(M_n)$ with $H_2(M_n, \Z)^{\langle f \rangle} \cong \Z$ is realizable by a Geiser or Bertini involution on a del Pezzo surface $(M_n, J)$. We emphasize that in fact, any automorphism $\varphi \in \Aut(M_n, J)$ (of \emph{any} complex structure $J$ on $M_n$) realizing such an involution $f$ is a Geiser or Bertini involution and $(M_n, J)$ is a del Pezzo surface. To see this, note that the tuple $(M_n, J, \varphi)$ forms a minimal $\langle f \rangle$-surface by Lemma \ref{irred-minimal}. Because $H_2(M_n, \Z)^{\langle f \rangle} = \Z$, work of Bayle--Beauville \cite[Theorem 1.4]{bayle--beauville} then shows that $(M_n, J)$ is a del Pezzo surface and that $\varphi$ is a Geiser involution with $n = 7$ or $\varphi$ is a Bertini involution with $n = 8$.
\end{remark}

\begin{remark}\label{rmk:order-2-irred}
There exists an irreducible class $f \in \Mod^+(M_7)$ of order $2$ that is \emph{not} conjugate in $\Mod(M_7)$ to a power of a Coxeter element of $W_7$. For example, the de Jonqui\'eres involution $\gamma$ (also considered in \Cref{Jonquiere--not--einstein}) defines an irreducible class $[\gamma] \in W_7$, but $[\gamma]$ is not conjugate to a Geiser involution in $\Mod(M_7)$ (cf. \cite[Proposition 3.14]{lee-involutions}). By \Cref{geiserbertini}, $[\gamma]$ is not conjugate to a power of a Coxeter element in $\Mod(M_7)$.
\end{remark}

Combining the above results concerning irreducible elements of prime order completes the proof of Theorem \ref{thm:primes-are-powers}.
\begin{proof}[Proof of \Cref{thm:primes-are-powers}]
    Let $f \in \Mod^+(M_n)$ be irreducible with odd, prime order. By Lemmas \ref{irredclassification} and \ref{negativeEven}, $f$ is conjugate in $\Mod^+(M_n)$ to a cuspidal conjugacy class of $W_n$. The theorem now follows from \Cref{primeIrredCuspidal}. 
\end{proof}

With \Cref{thm:primes-are-powers} in hand, we conclude with a proof of \Cref{primenielsen}.
\begin{proof}[Proof of Corollary \ref{primenielsen}]
For any $0 \leq n \leq 8$, the class $-I_n \in \Mod(M_n)$ acting by negation on $H_2(M_n, \Z)$ is contained in the center of $\Mod(M_n)$, and 
\[
    \Mod(M_n) = \langle \Mod^+(M_n), \, -I_n \rangle.
\]
Because $f$ has odd order, $f$ must be contained in $\Mod^+(M_n)$. 

First, suppose that $f$ is irreducible. If $3 \leq n \leq 8$ then $f$ is conjugate in $\Mod(M_n)$ to a power of a Coxeter element of $W_n \subset \Mod(M_n)$ by \Cref{thm:primes-are-powers}. By \Cref{coxetermainthm}, $f$ is realizable by a complex automorphism of a del Pezzo surface $(M_n, J)$. On the other hand, $\Mod(M_0) \cong \Z/2\Z$ and $\Mod(M_1) \cong (\Z/2\Z)^2$, and so there does not exist any irreducible elements $f \in \Mod(M_n)$ of order $p$ if $n =0$ or $n = 1$. Finally, \cite[Lemma 2.6(2)]{lee-involutions} shows that irreducible elements of $\Mod^+(M_2)$ are conjugate in $\Mod(M_2)$ to an element of $W_2 \cong \Z/2\Z$. Therefore, there does not exist any irreducible elements $f \in \Mod^+(M_n)$ of order $p$ if $n = 2$.

Suppose that $f \in \Mod^+(M_n)$ is reducible and write
\[
	f = (f_1, f_2) \in \Aut(H_2(M), Q_M) \times \Aut(H_2(\# k \overline{\CP^2}), Q_{\# k \overline{\CP^2}})
\]
for some $k > 0$ and some del Pezzo manifold $M$, for which $f_1$ is irreducible. 

In what follows, we will choose an order-$p$ diffeomorphism $\varphi_1 \in \Diff^+(M)$ with $[\varphi_1] = f_1$ so that $\Fix(\varphi_1) \neq \emptyset$. 
\begin{itemize}
\item Suppose that $f_1 = \Id$. Then $H_2(M, \Z)$ contains no $(-1)$-classes, and so $M \cong \CP^2$ or $M \cong \CP^1 \times \CP^1$. If $M \cong \CP^2$ then let $\varphi_1 = \mathrm{diag}(\zeta_p, 1, 1) \in \PGL_3(\C)$. If $M \cong \CP^1 \times \CP^1$ then let $\varphi_1 = (\mathrm{diag}(\zeta_p, 1), \Id) \in \PGL_2(\C) \times \PGL_2(\C)$. In either case, the set $\Fix(\varphi_1)$ is nonempty. 
\item Suppose that $f_1 \neq \Id$, and hence $f_1$ has order $p$. Then because $f_1$ is irreducible, $M \ncong M_m$ for $0 \leq m \leq 2$ by the same argument as above. Moreover, $M \ncong \CP^1 \times \CP^1$ because $\Mod(\CP^1 \times \CP^1) \cong (\Z/2\Z)^2$, which does not contain any elements of odd order. We now conclude that $M \cong M_m$ for some $3 \leq m \leq 8$. 

Because $f$ is contained in $\Mod^+(M_n)$, the irreducible component $f_1$ must be contained in $\Mod^+(M_m)$. By \Cref{irredclassification}, either
\[
    H_2(M_m, \Z)^{\langle f_1 \rangle} = \Z\{K_{M_m}\}, \quad \text{ or } \quad H_2(M_m, \Z)^{\langle f_1 \rangle} = \Z\{K_{M_m}, \, H - E_1\}
\]
up to conjugacy in $\Mod^+(M_m)$. In the latter case, $f_1$ is conjugate to an element of $P_m \subset \Mod^+(M_m)$ generating a cuspidal conjugacy class, which has even order by \Cref{negativeEven}. Therefore, $H_2(M_m, \Z)^{\langle f_1 \rangle} = \Z\{K_{M_m}\}$ because $f_1$ has odd, prime order $p$, and the signature of the lattice $(H_2(M_m, \Z)^{\langle f_1 \rangle}, Q_{M_{m}})$ is $1$.

Let $\varphi_1$ be an order-$p$ automorphism of a del Pezzo surface $(M,J)$ with $[\varphi_1]=f_1$, which exists by \Cref{thm:primes-are-powers} and \Cref{coxetermainthm}. By the $G$-signature theorem for $G = \Z/p\Z$ (e.g. see \cite[Theorem 3.1]{farb--looijenga}),
\[
	p - \sigma(M_m) = \sum_{z} \text{def}_z + \sum_{C} \text{def}_C.
\]
Because $\sigma(M_m) = 1-m < 0$, the set $\Fix(\varphi_1)$ is nonempty. 
\end{itemize}

Let $B = \{e_1, \dots, e_k\}$ denote a standard orthonormal $\Z$-basis of $H_2(\# k \overline{\CP^2})$ on which $f_2$ acts. Let $l$ be the number of $f_2$-orbits of $B$ of size $p$. After permuting the elements of $B$, we may assume that these orbits are of the form 
\[
    \{e_{i + pj} : 1 \leq i \leq p, \, 0 \leq j \leq l\}.
\]

Pick points $q_1, \dots, q_l \in M - \Fix(\varphi_1)$ with pairwise disjoint orbits (under the action of $\varphi_1$) in $M$. Then $\varphi_1$ induces an order-$p$ complex automorphism $\tilde{\varphi}_1$ of 
\[
	M' := \Bl_{\{\varphi_1^i(q_j) : i \in \Z, \, 1 \leq j \leq l\}} M
\]
such that $\Fix(\tilde{\varphi}_1) \neq\emptyset$. 

Suppose that $k > lp$ so that there exist elements of $B$ fixed by $f_2$. Note that any order-$p$ automorphism $\Phi$ of any complex surface $X$ with a fixed point $q \in \Fix(\Phi)$ induces an order-$p$ automorphism $\tilde\Phi$ on $\Bl_qX$. Moreover, $\Phi$ acts on the exceptional divisor $E$ over $q$ by an automorphism of finite order, and hence $\Fix(\tilde\Phi) \neq \emptyset$ and contains two points of $E$. Hence we form an $(k-lp)$-iterated blow up $X$ of $M'$ with an order-$p$ automorphism $\Phi$ of $X$ preserving each of the $(k-lp)$-many new exceptional divisors. The $\Z$-span of these new exceptional divisors in $H_2(X;\Z)$ forms a lattice $L$ isomorphic to $(\Z^{k-lp}, (k-lp)\langle -1 \rangle)$ which is fixed by $\Phi$.

Finally, identify the exceptional divisor $E_{i + pj}$ over $\varphi_1^i(q_j) \in M'$ with the class $e_{i + pj} \in B$ and isometrically identify $L$ with $\Z\{e_{lp+1}, \dots, e_k\}$. By construction, the action of $\Phi$ on $H_2(X; \Z)$ agrees with that of $f$.
\end{proof}

%
%

\footnotesize
\bibliographystyle{alpha} 
\bibliography{CyclicNielsen}

\bigskip
\noindent
Seraphina Eun Bi Lee \\
Department of Mathematics \\
University of Chicago \\
\href{mailto:seraphinalee@uchicago.edu}{seraphinalee@uchicago.edu}

\bigskip
\noindent
Tudur Lewis \\
School of Mathematics \\
University of Bristol \\
\href{mailto:et24788@bristol.ac.uk}{et24788@bristol.ac.uk}

\bigskip
\noindent
Sidhanth Raman \\
Department of Mathematics \\
University of California, Irvine \\
\href{mailto:svraman@uci.edu}{svraman@uci.edu}

\end{document}